\documentclass{article} 
\usepackage[T1]{fontenc}
\usepackage{lmodern}
\usepackage{microtype}

\usepackage{amsthm,amsfonts,amssymb,amsmath}
\usepackage{bookmark}
\usepackage{calligra}
\usepackage{chngcntr}
\usepackage{enumitem}
\usepackage[a4paper]{geometry}
\usepackage{indentfirst}

\usepackage{mathrsfs}
\usepackage{mathtools}
\usepackage{stmaryrd}
\usepackage{tikz}
\usepackage{tikz-3dplot}
\usepackage{tikz-cd}
\usepackage[nottoc]{tocbibind}
\usepackage{verbatim}
\usepackage[all]{xy}

\usepackage{hyperref}
\usepackage[capitalise]{cleveref}
\hypersetup{
	linktoc=page,
	colorlinks = true,
	linkcolor = [rgb]{0.28,0.41,0.19}, 
	urlcolor = [rgb]{1, 0.5, 0.5}, 
	citecolor = [rgb]{ 0.75, 0.2, 0.33} 
}

\newtheorem{thm}{Theorem}[subsection] 
\newtheorem{cor}[thm]{Corollary}
\newtheorem{lem}[thm]{Lemma}  
\newtheorem{prop}[thm]{Proposition}
 
\newtheorem{theorem}[thm]{Theorem}

\theoremstyle{definition}

\newtheorem{defn}[thm]{Definition}
\newtheorem{rem}[thm]{Remark}
\newtheorem{secnumber}[thm]{}

\numberwithin{equation}{thm}

\newcommand{\bb}[1]{\mathbb{#1}}
\newcommand{\rr}[1]{\mathrm{#1}}
\newcommand{\cc}[1]{\mathcal{#1}}
\newcommand{\scr}[1]{\mathscr{#1}}

\newcommand{\cros}{^{\times}}
\newcommand{\dR}{\mathrm{dR}}

\newcommand{\gal}{\mathrm{Gal}}
\newcommand{\gr}{\mathrm{gr}}
\newcommand{\Hdr}[1]{\mathrm{H}^{#1}_{\mathrm{dR}}}
\newcommand{\HH}{\mathrm{H}}

\newcommand{\Hyp}{\mathcal{H}yp}
\newcommand{\im}{\mathrm{im}}
\newcommand{\inv}{^{-1}}
\newcommand{\point}{\mathcal{P}}
\newcommand{\pr}{\mathrm{pr}}
\newcommand{\tor}{\mathrm{tor}}

\DeclareMathOperator{\ord}{ord}
\DeclareMathOperator{\id}{id}

\DeclareMathOperator{\HP}{HP}

\DeclareMathOperator{\Tr}{Tr}

\DeclareMathOperator{\Nm}{Nm}

\DeclareMathOperator{\CHyp}{Hyp}
\DeclareMathOperator{\irr}{irr}

\let\mod\Mod
\let\emptyset\varnothing
\let\oplus\bigoplus
\let\tilde\widetilde

\def\Gm{\mathbb{G}_m}
\def\DHyp{\mathscr{H}yp}
\def\dd{\mathbf{d}}

\newcommand{\quash}[1]{}

\newcommand\scalemath[2]{\scalebox{#1}{\mbox{\ensuremath{\displaystyle #2}}}}

\title{Irregular Hodge filtration of hypergeometric differential equations}
\begin{document}

\author{Yichen Qin, Daxin Xu}
\maketitle

\begin{abstract}
	Fedorov and Sabbah--Yu calculated the (irregular) Hodge numbers of hypergeometric connections. 
	In this paper, we study the irregular Hodge filtrations on hypergeometric connections defined by rational parameters, and provide a new proof of the aforementioned results. 
	Our approach is based on a geometric interpretation of hypergeometric connections, which enables us to show that certain hypergeometric sums are everywhere ordinary on $|\mathbb{G}_{m,\mathbb{F}_p}|$ (i.e. ``Frobenius Newton polygon equals to irregular Hodge polygon''). 
\end{abstract}

\section{Introduction} \label{s:intro}

The primary focus of this article is to investigate the Hodge theoretic properties of confluent hypergeometric differential equations. These differential equations have irregular singularities and are equipped with \textit{irregular Hodge filtrations}, constructed in \cite{Sabbah18irregular}.
The irregular Hodge theory, initiated by Deligne \cite{Del07}, extends the classical Hodge theory and has been developed in the works of Sabbah, Yu, Esnault--Sabbah--Yu, Sabbah--Yu \cite{Sabbah-2010-laplace,yu2012irregular,esnault17e1, Sabbah-Yu-2015, Sabbah18irregular}. 

Let $n\geq m$ be two non-negative integers, $\lambda$ a real number, and $\alpha=(\alpha_1,\ldots,\alpha_n)$ and $\beta=(\beta_1,\ldots,\beta_m)$ two non-decreasing sequences of real numbers in $[0,1)$. 
Let $S$ be the scheme $\bb{G}_m\backslash\{1\}$ (resp. $\bb{G}_m$) if $n=m$ (resp. $n>m$) with coordinate $z$. The \textit{hypergeometric equation} is the linear differential equation defined by the differential operator
		\begin{equation}\label{eq:Hyp-eq}
			\begin{split}
				\rr{Hyp}_\lambda(  \alpha ;  \beta ):= \lambda \prod_{i=1}^n (z\partial_z-\alpha_i)-  z\prod_{j=1}^m (z\partial_z-\beta_j).
			\end{split}
		\end{equation}
	The \textit{hypergeometric connection} $\Hyp_\lambda(\alpha;\beta)$ is the associated connection on the complex algebraic variety $S_{\mathbb{C}}$, see \eqref{eq:Hyp-conn}.  
We say that the pair $(\alpha, \beta)$ is \textit{non-resonant} if $\alpha_i\neq \beta_j$ for any $i$ and $j$. 
In this case, the hypergeometric connection $\Hyp_\lambda(\alpha;\beta)$ is irreducible and rigid, as seen by combining the works of Beukers--Heckman \cite{Beukers-Heckman} and Katz \cite{katz1990exponential}.

When $n=m$, hypergeometric connections have regular singularities at $0,1$, and $\infty$. 
Simpson demonstrated that rigid irreducible connections on curves with regular singularities, whose eigenvalues of monodromy actions at singularities have norm $1$, underlie complex variations of Hodge structure \cite[Cor.\,8.1]{simpson90harmonic}.
In this case, Fedorov computed the Hodge numbers associated with the Hodge filtrations of irreducible hypergeometric connections \cite{fedorov2017hypergeometric}. An alternative proof was given by Martin \cite{martin2018}. 

When $n>m$, hypergeometric connections are called \textit{confluent}, indicating the merging of singularities, and have a regular singularity at $0$ and an irregular singularity at $\infty$. 
Sabbah showed in \cite[Thm.\,0.7]{Sabbah18irregular} that a rigid irreducible connection on $\mathbb{P}^1$ with real formal exponents at each singular point admits a variation of irregular Hodge structures away from singularities. 
For hypergeometric connections, Sabbah and Yu computed the corresponding irregular Hodge numbers \cite{sabbah-hypergeometric}. 
In addition, Castaño Domínguez--Sevenheck \cite[Thm.\,4.7]{DS19} and Castaño Domínguez--Reichelt--Sevenheck \cite[Thm.\,5.8]{DRS19} explicitly calculated the irregular Hodge filtration for $m=0$ or $1$, respectively. 

In this article, we focus on the case where $\lambda$, $\alpha$, and $\beta$ are rational numbers. 
We explicitly construct the irregular Hodge filtration $F_{\irr}^{\bullet}$ on hypergeometric connections in \cref{thm:Hodge-fil} and provide a uniform method for reproving theorems of Fedorov and Sabbah--Yu. 

\begin{thm}[\ref{thm:Hodge-fil}]\label{thm:Hodge-intro} 
	Suppose $(\alpha,\beta)$ is non-resonant. 
	We define a map $\theta\colon \{1,\ldots,n\}\to \bb{R}$ by
		\begin{equation}\label{eq:theta(k)}
			\theta(k)= (n-m)\alpha_k+\#\{i\mid \beta_i< \alpha_k\}+(n-k)-\sum_{i=1}^n \alpha_i+\sum_{j=1}^m\beta_j. 
		\end{equation}
		Then, up to an $\bb{R}$-shift\footnote{Our Hodge numbers $\theta(k)$'s are normalized according to the geometric interpretation in Proposition~\ref{prop::geometric-realization}, and is different from those of Fedorov and Sabbah--Yu by a shift.}, the jumps of the irregular Hodge filtration on $\Hyp_\lambda(\alpha,\beta)$ occur at $\theta(k)$ and for any $p\in \bb{R}$ we have 
	$$\rr{rk}\,\mathrm{gr}^p_{F_{\irr}}\Hyp_\lambda(\alpha;\beta) = \#\theta\inv(p).$$
\end{thm}

\subsection{Application to Frobenius slopes of hypergeometric sums}

	Our method has an arithmetic application to the Frobenius slopes of hypergeometric sums, the arithmetic incarnation of hypergeometric connections \cite{katz1990exponential}. 

	Let $K$ be a $p$-adic field with residue field $k=\mathbb{F}_p$ containing an element $\pi$ satisfying $\pi^{p-1}=-p$. 
	Such an element $\pi$ corresponds to an additive character $\psi:k\to K^{\times}$ by Dwork's theory \cite{Dw74}. 
	Suppose that $(\alpha,\beta)$ is non-resonant and that $\alpha_i=\frac{a_i}{p-1},\beta_j=\frac{b_j}{p-1}$ are in $\frac{1}{p-1}\mathbb{Z}$.  
	Miyatani \cite{Miy} showed that there exists a unique Frobenius structure $\varphi$ (up to a scalar) on the analytification of hypergeometric connection $\Hyp_{(-1)^{m+np}/\pi^{n-m}}(\alpha;\beta)$ over $S_K$, which underlies an overconvergent $F$-isocrystal on $S_k$ (called the \textit{hypergeometric $F$-isocrystal}). 
	The Frobenius trace of $\varphi$ at a point $a\in S(\mathbb{F}_{q})$ is given by the \textit{hypergeometric sum} $\CHyp(\alpha;\beta)(a)$, defined by
	\begin{equation*}
	\sum_{\begin{subarray}{c} x_i, y_j\in \mathbb{F}_{q}^{\times}, \\ x_1\cdots x_n=ay_1\cdots y_m\end{subarray}} \psi\biggl( \Tr_{\mathbb{F}_{q}/k} \biggl(\sum_{i=1}^n x_i -\sum_{j=1}^m y_j\biggr)\biggr) \cdot 
		\prod_{i=1}^n \omega^{a_i}(\Nm_{\mathbb{F}_{q}/k}(x_i))
		\prod_{j=1}^m \omega^{-b_j}(\Nm_{\mathbb{F}_{q}/k}(y_j)),
	\end{equation*}
	where $\omega:\mathbb{F}_p^{\times} \to K^{\times}$ denotes the Teichmüller lift. 

	Frobenius eigenvalues of $\varphi$ at $a$ are Weil numbers and have complex absolute valuations $p^{\frac{n+m-1}{2}}$ via an isomorphism $\overline{K}\simeq \mathbb{C}$. 
	One intriguing question to explore is the investigation of their $p$-adic valuations, as there is an anticipated connection between these valuations and the irregular Hodge numbers of hypergeometric connections.
	Following \cite{Mazur}, we encode the information of $p$-adic valuations of Frobenius eigenvalues and irregular Hodge numbers into Newton polygon and Hodge polygon respectively. 
	Our construction allows us to show the following result. 
	
\begin{thm}[\ref{t:Frobslope}]\label{thm:NPHP-intro}
	Suppose $n>m$, $(\alpha,\beta)$ is non-resonant and $\alpha_i,\beta_j$ lie in $\frac{\mathbb{Z}}{p-1}\cap [0,1)$. 
	For every $a\in \Gm(\mathbb{F}_q)$, the Frobenius Newton polygon of $\CHyp(\alpha;\beta)(a)$ (normalized by $\ord_q$) coincides with the irregular Hodge polygon defined by the associated irregular Hodge numbers \eqref{eq:theta(k)}. 
\end{thm}

For crystalline cohomology groups of a smooth proper variety over $k$, Mazur and Ogus showed that the associated (Frobenius) Newton polygon lies above the Hodge polygon defined by Hodge numbers \cite{Mazur,BO}. 
For $F$-crystals associated with exponential sums, ``Newton above Hodge'' results were studied by Dwork's school. 
Dwork, Sperber and Wan \cite{Dw74, Sp77, Wan93} proved that Kloosterman sums (hypergeometric sums of type $(n,0)$ with $\alpha=(0,\dots,0)$) are everywhere ordinary on $|\mathbb{G}_{m,\mathbb{F}_p}|$ (i.e. two polygons coincide for every closed point $a\in |\Gm|$). 
Our proof is based on the works of Adolphson--Sperber \cite{AS89, AS93} and a criterion for ordinariness due to Wan \cite{Wan93}. 

\begin{rem}
	(i) One may also consider the Frobenius Newton polygon of hypergeometric sums defined by multiplicative characters of orders dividing $p^s-1$ for a positive integer $s$, and we still expect that the associated Frobenius Newton polygon lies above the irregular Hodge polygon. 
	However, the associated hypergeometric sums may not be ordinary in the case $s>1$. 
	There is an example of hypergeometric sums (of type $(n,m)=(2,0)$), for which the Frobenius Newton polygon lies strictly above the irregular Hodge polygon for every $a\in |\mathbb{G}_{m,\mathbb{F}_p}|$ \cite{adolphson1987twisted}. 

	(ii) The ordinariness of hypergeometric sums fails in the non-confluent case (i.e., $n=m$) as well. 
	For $p=31$, and the hypergeometric sum defined by $\alpha=(0,0,0,0)$, $\beta=(\frac{1}{5},\frac{2}{5},\frac{3}{5},\frac{4}{5})$ at $a=4$ or $17$, 
	its Newton polygon (with slope $(\frac{5}{2},\frac{5}{2},\frac{9}{2},\frac{9}{2})$) \cite[Appendix A.5]{DK17}\footnote{In \textit{loc. cit}, the Frobenius slopes are normalized and are different from our convention by a shift of $2$.} strictly lies above the irregular Hodge polygon (with slope $(2,3,4,5)$). 
\end{rem}

\subsection{Strategy of proof}\label{s:idea}

The proof of Theorem~\ref{thm:Hodge-intro} can be reduced to calculating the irregular Hodge filtration on each fiber of $\Hyp_\lambda(\alpha,\beta)$. To achieve this, we adopt an approach similar to those used in \cite{fresan2018hodge,Sabbah23airy,Qin23Hodge}, where the authors calculated the Hodge numbers of motives attached to Kloosterman and Airy moments. The key ingredient of this argument is an (exponentially) geometric interpretation of hypergeometric connections in Proposition~\ref{label::geom-3}. More precisely, there exists a smooth quasi-projective variety $X$ with a regular function $g\colon X\times S\to \mathbb{A}^1$, such that the hypergeometric connections are subquotients of the $\mathscr{D}_S$-module $\mathcal{H}^{N}\pr_{+}(\mathcal{O}_{X\times S},\rr{d}+\rr{d}g)$, where $N=\dim X$  and $\pr$ is the projection $\pr\colon X\times S \to S$. 
Our construction is motivated by Katz's hypergeometric sums and the function-sheaf dictionary. A related construction can be found in \cite{KY21hypergeometric}.

Through this geometric interpretation, each fiber $\Hyp_\lambda(\alpha,\beta)_a$ at $a\in S(\mathbb{C})$ is identified with a subquotient of the twisted de Rham cohomology of the pair $(X,g_a:=g\mid_{\pr_z\inv(a)})$, i.e., the hypercohomology of the twisted de Rham complex $(\Omega^\bullet_{X},\rr{d}+\rr{d}g_a)$. Then we reduce to calculate the irregular Hodge filtration on the twisted de Rham cohomology of the pair $(X,g_a)$ (up to a shift).

The irregular Hodge filtration on the twisted de Rham cohomology of the pairs $(X,g_a)$ has been studied by Yu \cite{yu2012irregular}. 
In the context of our case, we can select $X=\mathbb{G}_m^{n+m-1}$ and $g_a$ as a Laurent polynomial with good properties, see Proposition~\ref{prop::geometric-realization-convolution}. 
Under these assumptions, Yu showed that the irregular Hodge filtration on $\mathrm{H}_{\rr{dR}}^{n+m-1}(X,g_a)$ can be calculated by the Newton polyhedron filtration on the Newton polytope $\Delta(g_a)$ \eqref{eq:newton-monomial}. 
This identification enables us to calculate via a combinatorial approach, leading to a fiber-wise version of Theorem \ref{thm:Hodge-intro} as follows:

\begin{thm}[\ref{thm::hodge-number}]\label{thm::hodge-intro2}
	Up to an $\bb{R}$-shift, the jumps of the irregular Hodge filtration $F^\bullet_{\rr{irr}}$ on the fiber $\Hyp(\alpha;\beta)_a$ occur at $\theta(k)$ from \eqref{eq:theta(k)} for $1\leq k\leq n$. Moreover, we have
		$\operatorname{dim}\mathrm{gr}^p_{F_{\irr}}\Hyp(\alpha;\beta)_a = \#\theta\inv(p)$ for any $p\in \bb{R}$.
\end{thm}

Moreover, our calculation allows us to answer a question of Katz \cite[6.3.8]{katz1990exponential} on the comparison between modified hypergeometric $\mathscr{D}$-modules and hypergeometric connections in the resonant case (see \cref{p:compareDmods}) when the parameters are rational.

\subsection{Organization of this article}
The article is organized as follows. In \cref{sec:hyp-conn}, we present a geometric interpretation of hypergeometric connections. 
\cref{sec::hodge-number} is devoted to the proof of \cref{thm::hodge-intro2}  and \cref{thm:Hodge-intro}. 
In \cref{sec:Frob-struc}, we study hypergeometric sums defined by multiplicative characters of orders dividing $p-1$ and prove that they are ordinary (\cref{thm:NPHP-intro}).

\section{Hypergeometric connections} \label{sec:hyp-conn}

In this section, we give an (exponentially) geometrical interpretation of the hypergeometric connections in Propositions~\ref{prop::geometric-realization-convolution} and~\ref{prop::geometric-realization}. We work with varieties over $\bb{C}$ in Sections~2 and 3.

\subsection{Review of hypergeometric connections}
We review properties of hypergeometric connections and of two modified hypergeometric $\scr{D}$-modules following \cite{katz1990exponential}.

\begin{secnumber} \textbf{Hypergeometric connections.}
	Let $n\geq m$ be two integers $\ge 0$, $\alpha=(\alpha_1,\ldots,\alpha_n)$ and $\beta=(\beta_1,\ldots,\beta_j)$ two sequences of non-decreasing rational numbers (and we don't require that they lie in $[0,1)$ as in \S~\ref{s:intro}), and $\lambda\in \mathbb{Q}$.
	Let $\mathscr{D}_S$ be the sheaf of differential operator on $S$ (\S~\ref{s:intro}). 
	Then, the hypergeometric connection $\Hyp_\lambda(\alpha;\beta)$ on $S$ is defined by \eqref{eq:Hyp-eq}
\begin{equation}\label{eq:Hyp-conn}
	\begin{split}
		\mathscr{D}_{S}/\rr{Hyp}_\lambda(\alpha;\beta).
	\end{split}
\end{equation} 

By \cite[(3.1)]{katz1990exponential}, one has for $\gamma\in \mathbb{Q}$ that
\begin{equation}\label{eq:reduce-to-alpha1=0}
		\begin{split}
			\Hyp_\lambda(\alpha;\beta)\otimes (\cc{O},\rr{d}+\gamma\tfrac{\rr{d}z}{z})\simeq \Hyp_\lambda(\alpha+\gamma;\beta+\gamma),
		\end{split}
	\end{equation}
where $\alpha+\gamma$ (resp. $\beta+\gamma$) is the sequence consisting of $\alpha_i+\gamma $ (resp. $\beta_j+\gamma)$. Furthermore, one has for $\mu\in \mathbb{Q}\cros$ that
\begin{equation}\label{eq:coeff-change}
	\begin{split}
		[x\mapsto \mu\cdot x]^+\Hyp_\lambda(\alpha;\beta)\simeq \Hyp_{\lambda/\mu}(\alpha;\beta).
	\end{split}
\end{equation}
Thanks to the above relations, we can often assume that $\lambda=1$ and $\alpha_1=0$. For simplicity, we denote by $\Hyp(\alpha;\beta)$ the connection $\Hyp_1(\alpha;\beta)$.

When the pair $(\alpha,\beta)$ is non-resonant, i.e., $\alpha_i-\beta_j\not\in \bb{Z}$ for any $i,j$, Katz showed in \cite[Prop.\,3.2]{katz1990exponential} that $\Hyp(\alpha;\beta)$ is irreducible, and only depends on $\alpha \mod \bb{Z}$ and $\beta\mod \bb{Z}$. In this case, we may assume that $\alpha$ and $\beta$ are two non-decreasing sequences of rational numbers in $[0,1)$.
\end{secnumber}

\begin{secnumber} \textbf{Modified hypergeometric \texorpdfstring{$\scr{D}$}{D}-modules.}
Given a morphism $g$ between smooth varieties, for bounded complex of holonomic algebraic $\mathscr{D}$-modules, following \cite[App.\,A.1]{fresan2018hodge}, we denote by $g^+$, $g_{+}$, and $g_{\dagger}$ the derived pullback functor, the pushforward functor, and the pushforward with compact support functor respectively.  The $k$-th cohomology of a complex $K$ is denoted by $\mathcal{H}^k(K)$.

Let $\mathrm{mult}\colon \bb{G}_m\times \bb{G}_m\to \bb{G}_m$ be the product map. The convolution functors $\star_*$ and $\star_!$ on $\mathbb{G}_m$ are defined, for two objects $M$ and $N$ of $\mathrm{D}^b(\scr{D}_{\bb{G}_m})$ by
	$$M\star_* N:= \mathrm{mult}_+(M\boxtimes N)
		\text{ and } 
	M\star_! N:= \mathrm{mult}_\dagger M\boxtimes  N $$
respectively. 
These convolution functors are associative and commutative. Moreover, the duality functor $\mathbb{D}$ interchanges $\star_!$ and $\star_*$.
\end{secnumber}
\begin{defn}
	Let $\alpha$ and $\beta$ be two sequences of rational numbers. For $?\in \{!,*\}$, the convolution
		$$\Hyp(\alpha_1;\emptyset)\star_?\cdots \star_?\Hyp(\alpha_n;\emptyset)\star_?\Hyp(\emptyset;\beta_1)\star_?\cdots \star_?\Hyp(\emptyset;\beta_m)$$
	is a holonomic $\mathscr{D}_{\mathbb{G}_m}$-module \cite[(6.3.6)]{katz1990exponential}.
	We denote it by $\Hyp(?;\alpha;\beta)$ and call it \textit{modified hypergeometric $\scr{D}$-module}.
\end{defn}

The above two modified hypergeometric $\scr{D}$-modules are not isomorphic to the hypergeometric connections in general. When ($\alpha,\beta$) is non-resonant, the natural map 
	\begin{equation}\label{eq:forget-support}
		\begin{split}
			\Hyp(!;\alpha;\beta)\to \Hyp(*;\alpha;\beta)
		\end{split}
	\end{equation}
is an isomorphism, as seen by using an argument similar to those in \cite[Thm.\,8.4.2(5)]{katz1990exponential} and \cite[Prop.\,3.3.3]{Miy}. 
In this case, both modified hypergeometric $\scr{D}_{\mathbb{G}_m}$-modules are isomorphic to the hypergeometric connection $\Hyp(\alpha;\beta)$ by \cite[(5.3.1)]{katz1990exponential}.

\subsection{The Newton polytope of a Laurent polynomial}
We study the Newton polytope of a Laurent polynomial appearing in the geometric interpretation of hypergeometric connections in Proposition~\ref{prop::geometric-realization}.

\begin{defn} \label{d:Delta} 
	Let $N$ be a positive integer and $g(z_1,\cdots,z_N)=\sum_{\tau\in \bb{Z}^N}c(\tau)z^\tau$ be a Laurent polynomial in variables $z_1,\ldots,z_N$, with $z^{\tau}=\prod_{i=1}^{N}z_i^{\tau_i}$ for $\tau=(\tau_1,\cdots,\tau_N)$.
	\begin{enumerate}[label=(\arabic*).]
		\item The support of $g$ is the subset $\mathrm{Supp}(g)=\{\tau\mid c(\tau)\neq 0\}$ of $\mathbb{Z}^{N}$.
		\item The \textit{Newton polytope} $\Delta(g)$ is the convex hull of the set $\mathrm{Supp}(g)\cup \{0\}$ in $\bb{R}^{N}$.
		\item The Laurent polynomial $g$ is called \textit{non-degenerate} with respect to $\Delta(g)$ (or simply non-degenerate) if for each face $\sigma \subset \Delta(g)$ not passing through $0$, the Laurent polynomial $ g_\sigma:=\sum_{\tau\in \sigma\cap \bb{Z}^N}c(\tau)z^{\tau}$ has no critical point in $(\bb{C}\cros)^N$. 
	\end{enumerate}
\end{defn}

Let $n\geq m\geq 0$ and $d\geq 1$ be three integers, and $f\colon \bb{G}_{m}^{n+m}\to \bb{A}^1$ the Laurent polynomial
	\begin{equation}\label{eq::Laurent-polynomial}
		f: (x_2,\ldots,x_n,y_1,\ldots,y_m,z)\mapsto \sum_{i=2}^n x_i^d- \sum_{j=1}^m y_j^d+z\cdot \frac{\prod_{j=1}^my_j^d}{\prod_{i=2}^nx_i^d},
	\end{equation}
and $\pr_z\colon \bb{G}_m^{n+m}\to \bb{G}_{m}$ the projection onto the $z$-coordinate. For $a\in \bb{C}\cros$, we set $f_a=f\mid_{\pr_z\inv(a)}$.

We denote by $\{u_i,v_j\}_{2\leq i\leq n, 1\leq j\leq m}$ the coordinates in $\bb{R}^{n+m-1}$, and identify a monomial $\prod_i x_i^{a_i}\cdot \prod_j y_j^{b_j}$ with a lattice point $(a_i,b_j)\in\bb{Z}^{n+m-1}\subset \bb{R}^{n+m-1}$.

\begin{lem}\label{lem::Newton-Polygon-ineq-m=0} Assume that $n>m=0$ and $a\in \bb{C}\cros$. 
	\begin{enumerate}[label=(\arabic*).]
		\item The Laurent polynomial $f_a$ is convenient, i.e., the origin is in the interior of $\Delta(f_a)$.
		\item The Newton polytope $\Delta(f_a)$ is defined by 
			\begin{equation}\label{eq:h_k-1}
			\quad h_{n+1}:=\sum_{i=2}^n u_i\leq d \quad
			\text{ and } \quad
				h_{i_0}:=\sum_{i=2}^nu_i - (n-m)u_{i_0}\leq d, \quad 2\leq i_0\leq n. 
			\end{equation}

		\item The Laurent polynomial $f_a$ is non-degenerate with respect to $\Delta(f_a)$.
	\end{enumerate}
\end{lem}
\begin{proof}
	(1) Let $P_i$ for $2\leq i\leq n$, and $R$ be the points in $\mathbb{Z}^{n-1}$ corresponding to $x_i^{d}$ and $1/\prod x_i^d$ respectively. Observe that $0$ is an interior point of the Newton polytope because $0=\frac{1}{n}(\sum P_i+R)$. 
	
	(2) A face $\sigma\subset \Delta(f_a)$ of dimension $n-2$ must pass through $n-1$ points among $\{P_i,R\}$. So either $R\not\in \sigma$ or there exists a $P_{i_0}\not\in \sigma$. In the first case, the face lies on the hyperplane defined by the equation $h_{n+1}=d$. In the latter case, the face lies on the hyperplane defined by the equations $h_{i_0}=d$.

	(3) Let $\sigma$ be a face which does not pass through $0$. Since the support of $f_a$ has $n$ points, it must pass through at most $n-1$ points in $\mathrm{Supp}(f_a)$. Let $I\subset \{2,\ldots,n\}$ be a subset of the indices. Then $f_{a,\sigma}$ is either 
	$$f_{a,\sigma}=\sum_{i\in I} x_i^d, \quad \textnormal{or}\quad 
	f_{a,\sigma}=\sum_{i\in I} x_i^d+\frac{a}{\prod_{i=2}^n x_i^d}, \quad |I|\leq n-2. $$
	 We can check that they are smooth on $\bb{G}_m^{n-1}$. So $f_a$ is non-degenerate.
\end{proof}

\begin{lem}\label{lem::Newton-Polygon-ineq-m-neq-0} Assume that $n>m\neq 0$ and $a\in \bb{C}\cros$.
	\begin{enumerate}[label=(\arabic*).]
		\item The cone $\bb{R}_{\geq 0}\cdot \Delta(f_a)$ is defined by 
			$$u_i+v_j\geq 0,\quad v_j\geq 0$$
		for $i=2,\ldots,n$ and $j=1,\ldots,m$,
		\item The Newton polytope $\Delta(f_a)$ is defined by 
			$$u_i+v_j\geq 0,\quad v_j\geq 0,\quad h_{n+1}:=\sum u_i+\sum v_j\leq d$$
		and 
			\begin{equation}\label{eq:h_k-2} 
				h_{i_0}:=\sum_{i}u_i +\sum_{j}v_j - (n-m)u_{i_0}\leq d, \quad 2\leq i_0\leq n. 
			\end{equation}

		\item The Laurent polynomial $f_a$ is non-degenerate with respect to $\Delta(f_a)$ \footnote{
	In \cite[Lem.\,3.6]{Betti-hypergeometric-}, there is an alternative way of proving that $f_a$ is non-degenerate in this setting. }.
	\end{enumerate}
\end{lem}
\begin{proof}
	Let $P_i$ and $Q_j$ be the points in $\mathbb{Z}^{n+m-1}$ corresponding to monomials $x_i^d$ and $y_j^d$  for $2\leq i\leq n$ and $1\leq j\leq m$ respectively, and $R$ the lattice point corresponding to  $\prod_{j=1}^m y_j^d/\prod_{i=2}^n x_i^d$. In this case, the origin $0$ is not an interior point of the Newton polytope. So $\Delta(f_a)$ has $(n+m+1)$-many vertices. To determine a face of dimension $n+m-2$, we need to choose $(n+m-1)$-many points among $\{P_i,Q_j,R\}$.

	(1) For the first part, it suffices to determine faces $\sigma\subset \Delta(f_a)$ with dimensions $n+m-2$ containing $0$. 
	\begin{itemize}
		\item If $\sigma$ does not pass through $R$, it contain $(n+m-2)$ distinct points in $\{P_i,Q_j\}$. In this case, $\sigma$ misses one point $Q_{j_0}$, and lies on the hyperplane $v_{j_0}=0$. Otherwise, $\sigma$ misses one point $P_{i_0}$. Hence, the hyperplane is given by the equation $u_{i_0}=0$. Therefore, $R$ and $P_{i_0}$ lie on the two sides of the hyperplane respectively, which is absurd.
		\item If $\sigma$ passes through $R$, it contain $(n+m-3)$ distinct points in $\{P_i,Q_j\}$. In this case, $\sigma$ has to miss one $P_{i_0}$ and one $Q_{j_0}$, and lies on the hyperplane $u_{i_0}+v_{j_0}=0$.  Otherwise, $\sigma$ misses two $P_{i_0}, P_{i'_0}$ or $Q_{j_0},Q_{j_0'}$. So $\sigma$ lies on the hyperplane $u_{i_0}-u_{i_0'}=0$ or $v_{j_0}-v_{j_0'}=0$. However, the points $P_{i_0}, P_{i'_0}$ or $Q_{j_0},Q_{j_0'}$ lie be on different sides of the hyperplane $u_{i_0}-u_{i_0'}=0$ or $v_{j_0}-v_{j_0'}=0$, which contradicts the definition of $\sigma$.
	\end{itemize}

	(2) For the second part, it suffices to determine faces of dimension $n+m-2$ that do not pass through the origin.	 
	\begin{itemize}
		\item If $R\not\in \sigma$, then $\sigma$ contains all points $P_i$ and $Q_j$. In this case, $\sigma$ lies on the hyperplane $\sum u_i+\sum v_j=d$.
		\item If $R\in \sigma$, then $\sigma$ contains $n+m-2$ points among $\{P_i,Q_j\}$. In this case, $\sigma$ misses one $P_{i_0}$, and lies on the hyperplane $h_{i_0}=d$. Otherwise, it misses one $Q_{j_0}$ and lies on the hyperplane $\sum_{i=2}^n u_i+\sum_{j=1}^m v_j +(n-m)v_{j_0}=d$. However, the points $0$ and $Q_{j_0}$ are on different sides of the hyperplane.
	\end{itemize}
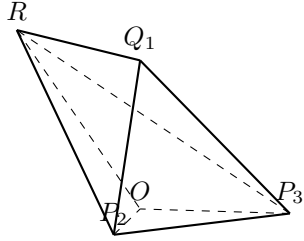
\begin{center}
\tdplotsetmaincoords{80}{100}
\begin{tikzpicture}[tdplot_main_coords]
 
	\draw[thick] (2,0,0) -- (0,2,0) node[anchor=south]{$P_3$};
	\draw[thick] (0,2,0) -- (0,0,2) node[anchor=south]{$Q_1$};
	\draw[thick] (0,0,2) -- (2,0,0) node[anchor=south]{$P_2$};
	\draw[thick] (2,0,0) -- (-2,-2,2) node[anchor=south]{$R$};
	\draw[dashed] (0,2,0) -- (-2,-2,2) node[anchor=south]{};
	\draw[thick] (0,0,2) -- (-2,-2,2) node[anchor=south]{};
	\draw[dashed] (-2,-2,2)--(0,0,0) node[anchor=south]{$O$};
	\draw[dashed] (2,0,0)--(0,0,0) node[anchor=south]{};
	\draw[dashed] (0,2,0)--(0,0,0) node[anchor=south]{};
\end{tikzpicture}
\end{center}

	(3) Let $\sigma$ be a face which does not pass through $0$. Since the support of $f_a$ has $n+m$ points, it must pass through at most $n+m-1$ points in $\mathrm{Supp}(f_a)$. Let $I\subset \{2,\ldots,n\}$ and $J\subset \{1,\ldots,m\}$ be two subsets of the indices. Then $f_{a,\sigma}$ is either 
	$$f_{a,\sigma}=\sum_{i\in I} x_i^d-\sum_{j\in J} y_j^d, \quad
	\textnormal{or} \quad f_{a,\sigma}=\sum_{i\in I} x_i^d-\sum_{j\in J} y_j^d+\frac{a\prod y_j^d}{\prod x_i^d}, \quad \textnormal{for}~ |I|+|J|\leq n+m-2. $$
	We can check that they are smooth on $\bb{G}_m^{n+m-1}$. So $f_a$ is non-degenerate.
\end{proof}

\begin{lem}\label{lem::Newton-Polygon-ineq-n=m} Assume that $n=m$ and $a\in \bb{C}\cros$.
	\begin{enumerate}[label=(\arabic*).]
		\item The cone $\bb{R}_{\geq 0}\cdot \Delta(f_a)$ is defined by 
			$$u_i+v_j\geq 0,\quad v_j\geq 0$$
		for $i=2,\ldots,n$ and $j=1,\ldots,m$,
	\item The Newton polytope $\Delta(f_a)$ is defined by  
			\begin{equation}\label{eq:h_k-3}
				u_i+v_j\geq 0,\quad v_j\geq 0, \quad \text{ and } \quad h_{n+1}:=\sum u_i+\sum v_j\leq d.
			\end{equation}

		\item The Laurent polynomial $f_a$ is non-degenerate with respect to $\Delta(f_a)$ if $a\neq 1$.
	\end{enumerate}
\end{lem}
\begin{proof}
	We use the same notation as in Lemma~\ref{lem::Newton-Polygon-ineq-m-neq-0}. The proof of the first assertion is the same as that in Lemma~\ref{lem::Newton-Polygon-ineq-m-neq-0}. The second assertion follows from the observation that the points $\{P_i,Q_j,R\}$ all lie on the hyperplane $\sum u_i+\sum v_j-d=0$. 
	
	Let $\sigma$ be the face passing through $\{P_i,Q_j,R\}$. If a face $\tau$ of $\Delta(f_a)$ does not contain $0$, it is a face of $\sigma$. One can check that if $\tau$ is a proper face of $\sigma$, there is no solution for the system of equations
		$$f_{a,\tau}=\partial_{x_i}f_{a,\tau} =\partial_{y_j}f_{a,\tau}=0.$$
	If $\tau=\sigma$, the system of equations
		$$f_a=\partial_{x_i}f_{a} =\partial_{y_j}f_{a}=0$$
	has solutions in $\bb{G}_m^{n+m-1}$ if and only if $a=1$. So $f_a$ is non-degenerate with respect to $\Delta(f_a)$ if $a\neq 1$.
\end{proof}
\begin{rem}
	The volume of $\Delta(f_a)$ is $\frac{d^{n+m-1}n}{(n+m-1)!}$. In fact, the Newton polytope can be decomposed into $n$-copies $n+m-1$-simplexes, and each of them has volume $\frac{d^{n+m-1}}{(n+m-1)!}$. 
\end{rem}

\subsection{Geometric interpretations}\label{sec:geom-intpn}
We present geometric interpretations of hypergeometric connections here. Let $d$ be a common denominator of $\alpha_i$ and $\beta_j$, and set $a_i=d\cdot \alpha_i$ and $b_j=d\cdot \beta_j$. To $\alpha_i$ (resp. $\beta_j$), we associate the character $\chi_i\colon \mu_d\to \bb{C}\cros$ (resp. $\rho_j$) which sends $\zeta_d$ to $\zeta^{a_i}_d$ (resp. $\zeta^{b_j}_d$). Set

	\begin{equation}\label{eq:char}
		\begin{split}
			\chi\times \rho=\chi_1\times \ldots\times \chi_n\times \rho_1\inv\times \ldots\times \rho_m\inv, \quad
			\tilde \chi\times \rho=\chi_2\times \ldots\times \chi_n\times \rho_1\inv \times \ldots\times \rho_m\inv
		\end{split}
	\end{equation}
as products of these characters.

Now we introduce two diagrams as follows:

\begin{itemize}
	\item Let $\bb{G}_m^{n+m}$ be the torus with coordinates $x_i, y_j$ for $1\leq i\leq n$ and $1\leq j\leq m$. The group $\mu_d^{n+m}$ acts on $\bb{G}_m^{n+m}$ by multiplication $d$-th roots of unity on coordinates $x_i$'s and $y_j$'s.  Then we consider the diagram  
		\begin{equation} \label{eq:diagram-Hyp-isotypic}
		\begin{tikzcd}
			& \bb{G}_{m}^{n+m}\ar[ld,"\sigma"']\ar[rd,"\varpi"] &	\\
			\bb{A}^1 &		& \bb{G}_{m}
		\end{tikzcd}
	\end{equation}
	where $\sigma(x_i,y_j)=\sum_{i=1}^n x_i^d-\sum_{j=1}^my_j^d$, and $\varpi(x_i,y_j)=\prod_{i=1}^n x_i^d/\prod_{j=1}^m y_j^d.$ 

 \item Let $\bb{G}_m^{n+m}$ be the torus with coordinates $z,x_i,y_j$ for $2\leq i\leq n$ and $1\leq j\leq m$, and $S$ be $\bb{G}_m$ (resp. $\bb{G}_m\backslash\{1\}$) if $n\neq m$ (resp. $n=m$). 
 The group $G=\mu_d^{n+m-1}$ acts on coordinates $x_i$'s and $y_j$'s by multiplication $d$-th roots of unity. Then we consider the diagram  
	 \begin{equation}\label{eq:diagram-Hyp-mot}
		\begin{tikzcd}
		 	& \bb{G}_{m}^{n+m}\ar[ld,"f"']\ar[rd,"\pr_z"] &U:=S\times \bb{G}_m^{n+m-1} \ar[l,hook']	\ar[rd,"\pr_z"]&\\
		 	\bb{A}^1 &		& \bb{G}_{m} & S\ar[l,hook']
	 	\end{tikzcd}
	\end{equation}
	where $\pr_z$ is the projection on the $z$-coordinate and $f$ is defined in \eqref{eq::Laurent-polynomial}.
\end{itemize}

Let $\mathcal{E}^{z}=(\mathcal{O},\mathrm{d} + \mathrm{d}z)$ be the exponential $\scr{D}$-module on $\bb{A}^1_z$. For a regular function $f\colon X\to \bb{A}^1_z$, we denote by $\mathcal{E}^f$ the connection $(\cc{O}_X,\rr{d}+\rr{d}f)$ on $X$.

\begin{prop}\label{prop::geometric-realization-convolution}
	Let $\alpha$ and $\beta$ be as above.
	\begin{enumerate}[label=(\arabic*).,ref=\ref{prop::geometric-realization-convolution}.(\arabic*)]
		\item \label{label::geom-1} The complex $\varpi_{?}\mathcal{E}^{\sigma}$ is concentrated in degree $0$ for $?\in \{\dagger,+\}$, and we have isomorphisms of $\mathscr{D}$-modules
			$$\Hyp(*;\alpha;\beta)\simeq (\varpi_{+}\mathcal{E}^{\sigma})^{(\mu_d^{n+m}, \chi\times \rho)}\ 
			\text{ and } \ 
			\Hyp(!;\alpha;\beta)\simeq (\varpi_{\dagger}\mathcal{E}^\sigma)^{(\mu_d^{n+m}, \chi\times \rho)},$$
		where the exponent $(\mu_d^{n+m}, \chi\times \rho)$ means taking the $\chi\times \rho$-isotypic component with respect to the action of $\mu_d^{n+m}$. 
		
	\item \label{label::geom-3}	If $\alpha_1=0$, we have 
		$$\Hyp(*;\alpha;\beta)\simeq( \cc{H}^0\pr_{z+}\mathcal{E}^{f})^{( G,\tilde\chi\times\rho)}\ 
		\text{ and } \ 
		\Hyp(!;\alpha;\beta)\simeq(\cc{H}^0\pr_{z\dagger}\mathcal{E}^{f})^{( G,\tilde\chi\times\rho)}.$$
	\end{enumerate} 
\end{prop}
\begin{proof}
	The case of $\Hyp(!;\alpha;\beta)$ can be deduced from the case of $\Hyp(*;\alpha;\beta)$ by applying the duality functor. So we only prove the latter case.

(1) Assume that $(n,m)=(1,0)$. Then $\sigma\colon \bb{G}_{m,x_1}\to \bb{A}^1$ is the map $x_1\mapsto x_1^d$ and $\varpi\colon \bb{G}_{m,x_1}\to \bb{G}_{m,z}$ is the $d$-th power map. So by the identity $\varpi_+\mathcal{O}_{\bb{G}_m}=\oplus_{i=0}^{d-1} \left(\mathcal{O}_{\bb{G}_m}, \mathrm{d}+\frac{i}{d}\frac{\mathrm{d}z}{z}\right)$ and the projection formula, we have
		$$(\varpi_+\mathcal{E}^{\sigma})
		=\mathcal{E}^{z}\otimes (\varpi_+\mathcal{O}_{\bb{G}_m})
		=\oplus_{i=0}^{d-1} \mathcal{E}^{z}\otimes\left(\mathcal{O}_{\bb{G}_m}, \mathrm{d}+\tfrac{i}{d}\tfrac{\mathrm{d}z}{z}\right),$$
	which is concentrated in degree $0$. Taking the isotypic component, we have
		\begin{equation*}
			\begin{split}
				(\varpi_+\mathcal{E}^{\sigma})^{(\mu_d^{n+m}, \chi\times \rho)}
				&=(\varpi_+\mathcal{E}^{x_1^d})^{(\mu_d,\chi_1)}=\mathcal{E}^{z}\otimes (\varpi_+\mathcal{O}_{\bb{G}_m})^{(\mu_d,\chi_1)}\\
				&=(\cc{O}_{\bb{G}_{m}},\rr{d}+\rr{d}z+\alpha_1\tfrac{\rr{d}z}{z})
				=\Hyp(*;\alpha_1;\emptyset)
			\end{split}
		\end{equation*}
	in the case where $(n,m)=(1,0)$. The proof of the case where $(n,m)=(0,1)$ is similar. In general, we use the induction on $n+m$. The proof follows from the following lemma.

\begin{lem}\label{lem::exterior-prod}
	Let $\alpha, \alpha , \beta$ and $\beta'$ be four sequences of rational numbers with common denominator $d$, whose lengths are $n,n',m$ and $m'$ respectively. We denote by $\chi_i,\chi_i',\rho_j,\rho_j'$ characters of $\mu_d$ corresponding to $\alpha_i,\alpha'_i,\beta_j,\beta_j'$ respectively. Let $\sigma$, and $\varpi$ (resp. $\sigma'$ and $\varphi'$) be the maps for $(n,m)$ (resp. $(n',m')$) in the diagram \eqref{eq:diagram-Hyp-isotypic}.
	
	Suppose that $(\varpi_+\mathcal{E}^{\sigma})$ and $(\varpi_+'\mathcal{E}^{\sigma'})$ are concentrated in degree $0$, and there are isomorphisms of $\scr{D}$-modules
		$$\Hyp(*;\alpha;\beta)
		\simeq (\varpi_+\mathcal{E}^{\sigma})^{(\mu_d^{n+m}, \chi\times \rho)}\ 
		\text{ and }\  
		\Hyp(*;\alpha';\beta')
		\simeq (\varpi_+\mathcal{E}^{\sigma'})^{(\mu_d^{n+m}, \chi'\times \rho')}.$$
	Then $((\varpi\cdot \varpi')_+\mathcal{E}^{\sigma\boxplus \sigma'})$ is also concentrated in degree $0$, and we have an isomorphism of $\scr{D}$-modules
			\begin{equation*}
				\begin{split}
					\Hyp(*;\alpha,\alpha';\beta,\beta')
					&\simeq ((\varpi\cdot \varpi')_+\mathcal{E}^{\sigma\boxplus \sigma})^{(\mu_d^{n+n'+m+m'}, \chi\times \chi'\times \rho\times \rho')}
				\end{split}
			\end{equation*}
		where $\varpi\cdot \varpi'=\mathrm{mult}\circ(\varpi\times \varpi')$, $\pr$ and $\pr'$ are the projections from $\bb{G}_m^{n+n'+m+m'}$ to $\bb{G}_m^{n+m}$ and $\bb{G}_m^{n'+m'}$ respectively, and $\sigma\boxplus \sigma'=\sigma\circ\pr+\sigma'\circ \pr'$ is the Thom-Sebastiani sum.
\end{lem}
\begin{proof}[Proof of Lemma~\ref{lem::exterior-prod}]
	This proof of this lemma is essentially that of \cite[Lem.\,5.4.3]{katz1990exponential}. Notice that the exterior product $\mathcal{E}^\sigma\boxtimes \mathcal{E}^{\sigma'}$ is $\mathcal{E}^{\sigma\boxplus \sigma'}$. Then
		\begin{equation*}
			\begin{split}
				(\varpi_+\mathcal{E}^\sigma)\star_*(\varpi'_+\mathcal{E}^{\sigma'})
				&=\mathrm{mult}_+((\varpi_+\mathcal{E}^{\sigma})\boxtimes (\varpi'_+\mathcal{E}^{\sigma'}))\\
				&=\mathrm{mult}_+(\varpi\times \varpi')_+ (\mathcal{E}^\sigma\boxtimes \mathcal{E}^{\sigma'})
				=(\varpi\cdot \varpi')_+\mathcal{E}^{\sigma\boxplus \sigma'}.
			\end{split}
		\end{equation*}
	By Künneth formula \cite[Prop.\,1.5.28(i) and Prop.\,1.5.30]{HTT-d-module}, we conclude that $(\varpi\cdot \varpi')_+\mathcal{E}^{\sigma\boxplus \sigma'}$ is again concentrated in degree $0$. We finish the proof by taking the corresponding isotypic components.
\end{proof}

(2)	Since $\alpha_1=0$, the character $\chi_1$ is trivial. So we have
	\begin{equation}\label{eq::identification}
		\begin{split}
			(\varpi_+\mathcal{E}^{\sigma})^{(\mu_d^{n+m}, \chi\times \rho)}
			&= \biggl(\Bigl(x_1\cdot \prod_{i=2}^nx_i^d/\prod_{j=1}^m y_j^d\Bigr)_+\mathcal{E}^{x_1+\sum_{i=2}^mx_i^d-\sum_jy_j^d}\biggr)^{(1\times G, 1\times \tilde\chi\times \rho)} \\
			&= (\pr_{z+}\mathcal{E}^{f})^{( G,\tilde\chi\times\rho)},
		\end{split}
	\end{equation}
	where we performed a change of variable $z=x_1\cdot \prod_{i=2}^nx_i^d/\prod_{j=1}^m y_j^d$ to get rid of the variable $x_1$ in the last isomorphism. Because $(\varpi_+\mathcal{E}^{\sigma})$ is concentrated in degree $0$, so is $(\pr_{z+}\mathcal{E}^{f})^{( G,\tilde\chi\times\rho)}$.
\end{proof}

\begin{cor}\label{cor::forget-support-iso}
	Assume that $(\alpha, \beta)$ is non-resonant. Then the natural map 
		$$ (\cc{H}^0\pr_{z\dagger}\mathcal{E}^{f})^{( G,\tilde\chi\times\rho)}\to  (\cc{H}^0\pr_{z+}\mathcal{E}^{f})^{( G,\tilde\chi\times\rho)}$$ 
	is an isomorphism of $\mathscr{D}_{\mathbb{G}_m}$-modules. In particular, for $a\in S(\bb{C})$, the forget-support map 
		$$\rr{H}^{n+m-1}_{\dR,c}(\mathbb{G}_{m}^{n+m-1},f_a)\to \Hdr{n+m-1}(\mathbb{G}_{m}^{n+m-1},f_a)$$
	is an isomorphism.
\end{cor}
\begin{proof}
	Using induction on the size of $\alpha$ and $\beta$, one can verify that the diagram
		$$\begin{tikzcd}
			\Hyp(!;\alpha;\beta) \ar[r,"\simeq"] \ar[d,"\simeq"]& \Hyp(*;\alpha;\beta) \ar[d,"\simeq"]\\
			(\cc{H}^0\varpi_\dagger\mathcal{E}^{\sigma})^{(\mu_d^{n+m}, \chi\times \rho)}\ar[r]\ar[d,"\simeq"] &(\cc{H}^0\varpi_+\mathcal{E}^{\sigma})^{(\mu_d^{n+m}, \chi\times \rho)}|_S\ar[d,"\simeq"]\\
			(\cc{H}^0\pr_{z\dagger}\mathcal{E}^{f})^{( G,\tilde\chi\times\rho)} \ar[r]& (\cc{H}^0\pr_{z+}\mathcal{E}^{f})^{( G,\tilde\chi\times\rho)}
		\end{tikzcd}$$
	is commutative, where the horizontal morphisms are the natural morphisms, the two upper vertical morphisms are those from Proposition~\ref{label::geom-1}, and the two lower vertical morphisms are \eqref{eq::identification}. So we deduce the isomorphism 
		$$(\cc{H}^0\pr_{z\dagger}\mathcal{E}^{f})^{( G,\tilde\chi\times\rho)}\to (\cc{H}^0\pr_{z+}\mathcal{E}^{f})^{( G,\tilde\chi\times\rho)}.$$
	At last, we take the non-characteristic inverse image along $a\colon \mathrm{Spec}(\mathbb{C})\to \mathbb{G}_m$, and the base change theorem \cite[Thm.\,1.7.3\,\&\,Prop.\,1.5.28]{HTT-d-module} to conclude the isomorphism of twisted de Rham cohomologies.
\end{proof}

\begin{rem}\label{rem::dual-mod-hypergeometric} 
			When ($\alpha,\beta$) is non-resonant, we deduce from Proposition~\ref{prop::geometric-realization-convolution} the isomorphism
				$$[z\mapsto (-1)^{n-m}z]^+\Hyp(\alpha,\beta)\simeq (\cc{H}^0\pr_{z+}\mathcal{E}^{-f})^{( G,\tilde\chi\times\rho)},$$
			by performing a change of variable by sending $x_i$ and $y_j$ to $-x_i$ and $-y_j$ respectively in the diagram \eqref{eq:diagram-Hyp-mot}. According to \eqref{eq:coeff-change}, the first term in the above is $\Hyp_{(-1)^{n-m}}(\alpha;\beta)$. In particular, the results in Corollary~\ref{cor::forget-support-iso} remain valid if we replace $f$ with $-f$.
\end{rem}

\subsection{Explicit cyclic vectors for hypergeometric connections}
We present explicit cyclic vectors for $\Hyp(\alpha;\beta)$ in terms of sections of some subquotients of some relative de Rham cohomology equipped with their Gauss--Manin connections. 
This point of view will be used in the computation of Hodge numbers in \cref{sec::hodge-number}.

Recall that $d$ is an integer such that $a_i=d\alpha_i$ and $b_j=d\beta_j$ are integers for all $i,j$, and we take notation from \eqref{eq:diagram-Hyp-mot}. 
When $(\alpha,\beta)$ is non-resonant and $
\alpha_1=0$, there exists an isomorphism between the hypergeometric connection $\Hyp(\alpha;\beta)$ and the relative de Rham cohomology $\mathcal{H}^{n+m-1}_{\mathrm{dR}}(U/S,f)^{(G,\tilde\chi\times \rho)}$ equipped with the Gauss--Manin connection by \cref{prop::geometric-realization-convolution}.

\begin{prop}\label{prop::geometric-realization}
	Suppose that $\alpha_1=0$ and $(\alpha, \beta)$ is non-resonant. 
	The relative de Rham cohomology $\mathcal{H}^{n+m-1}_{\mathrm{dR}}(U/S,f)^{(G,\tilde\chi\times \rho)}$ admits a cyclic vector, defined by the cohomology class of the differential form 
	$$
		\omega=\prod_{i=2}^n x_i^{a_i}\cdot \prod_{j=1}^m y_j^{-b_j} \frac{\mathrm{d}x_2}{x_2}\cdots\frac{\mathrm{d}x_n}{x_n} \frac{\mathrm{d}y_1}{y_1}\cdots \frac{\mathrm{d}y_m}{y_m}.
	$$ 
\end{prop}

\begin{rem} \label{r:cyclic-vector}
	Under the above assumption, the isomorphism class of $\Hyp(\alpha;\beta)$ depends only on the congruent classes of $\alpha,\beta$ modulo $\mathbb{Z}$. 
	Then, any differential form 
	\[
		\omega=\prod_{i=2}^n x_i^{u_i}\cdot \prod_{j=1}^m y_j^{-v_j} \frac{\mathrm{d}x_2}{x_2}\cdots\frac{\mathrm{d}x_n}{x_n} \frac{\mathrm{d}y_1}{y_1}\cdots \frac{\mathrm{d}y_m}{y_m},
	\]
	satisfying $u_i\equiv a_i, v_j\equiv b_j$ modulo $d$, is a cyclic vector of $\mathcal{H}^{n+m-1}_{\mathrm{dR}}(U/S,f)^{(G,\tilde\chi\times \rho)}$.
\end{rem}
\begin{proof}
	The morphism $\pr_z:\colon U\to S$ is smooth \eqref{eq:diagram-Hyp-mot}. It follows that the relative de Rham cohomologies $\mathcal{H}^{i}_{\mathrm{dR}}(U/S,f)$ are equipped with the Gauss-Manin connections $D:=\nabla_{z\partial_z}$, given by
	\begin{equation}\label{eq::Gauss-Manin}
		\nabla_{z\partial_z}\omega=z\partial_z\omega+ z\partial_z(f)\omega
	\end{equation}
	for $0\leq i\leq n+m-1$. By Lemmas~\ref{lem::Newton-Polygon-ineq-m=0}--\ref{lem::Newton-Polygon-ineq-n=m}, the Laurent polynomial $f_a:=f|_{\pr_z\inv(a)}$ is non-degenerate for each $a\in S(\bb{C})$. By \cite[Thm.\,1.4 and Thm.\,4.1]{adolphson1997twisted}, the cohomology group $\mathcal{H}^{i}_{\mathrm{dR}}(U/S,f_a)$ vanishes if $i\neq n+m-1$. 

	Now we consider the $(G,\tilde\chi\times \rho)$-isotypic component of the connection $\mathcal{H}^{n+m-1}_{\mathrm{dR}}(U/S,f)$, which can be identified with $(\mathcal{H}^0\pr_{z+}\mathcal{E}^{f})^{(G,\tilde\chi\times \rho)}$. It remains to prove that the cohomology class defined by the differential form 
		$$
		\omega=\prod_{i=2}^n x_i^{a_i}\cdot \prod_{j=1}^m y_j^{-b_j} \frac{\mathrm{d}x_2}{x_2}\cdots\frac{\mathrm{d}x_n}{x_n} \frac{\mathrm{d}y_1}{y_1}\cdots \frac{\mathrm{d}y_m}{y_m}
		$$ 
	is a cyclic vector for $\mathcal{H}^{n+m-1}_{\mathrm{dR}}(U/S,f)^{(G,\tilde\chi\times \rho)}$.

\begin{lem}\label{lem::computation-GaussManin}
	Let $t_2,\cdots,t_n,s_1,\cdots,s_m$ be integers and set
		$$
			\tilde \omega:=\prod_{i=2}^n x_i^{t_i}\cdot \prod_{j=1}^m y_j^{s_j} \frac{\mathrm{d}x_2}{x_2}\cdots\frac{\mathrm{d}x_{n}}{x_{n}} \frac{\mathrm{d}y_1}{y_1}\cdots \frac{\mathrm{d}y_m}{y_m}
		$$
	as a class in $\mathcal{H}^{n+m-1}_{\mathrm{dR}}(U/S,f)$. 
	For each $i$ and $j$ such that $2\leq i\leq n$ and $1\leq j\leq m$ respectively, we have
		$$
			(D-t_i/d)\tilde \omega=x_i^d\cdot \tilde \omega \quad \text{and} \quad  (D+s_j/d)\tilde \omega=y_j^d\cdot \tilde \omega.
		$$
\end{lem}
\begin{proof}
	We prove the identity for $(D-t_2)\tilde \omega$. And the proofs for the rest are identical. By \eqref{eq::Gauss-Manin}, we have $D\,\tilde \omega= \tfrac{z \cdot \prod_j y_j^d}{\prod_i x_i^d}\,\tilde \omega.$ Then by the definition of the relative twisted de Rham cohomology
		\begin{equation*}
			\begin{split}
				0&= \nabla_{U/S}\left(\prod_{i=2}^n x_i^{t_i}\cdot \prod y_j^{s_j} \frac{\mathrm{d}x_3}{x_3}\cdots\frac{\mathrm{d}x_{n}}{x_{n}} \frac{\mathrm{d}y_1}{y_1}\cdots \frac{\mathrm{d}y_m}{y_m}\right)
				=t_2\cdot \tilde \omega+ x_2\cdot\partial_{x_2}f \cdot \tilde \omega \\
				&=t_2\cdot \tilde \omega+ x_2\cdot\biggl(dx_2^{d-1}-dx_2\inv \frac{z\prod y_j^d}{\prod x_i^d}\biggr) \,\tilde \omega 
				=d(x_2^d-(D-t_2/d))\,\tilde \omega.
			\end{split} 
		\end{equation*}
	This is exactly what we want to prove.
\end{proof}

	We show that $\omega$ satisfies the hypergeometric differential equation $\rr{Hyp}(\alpha;\beta)$.
	By Lemma~\ref{lem::computation-GaussManin}, we have 
		$$
			\prod_{i=2}^n (D-\alpha_i)\omega  =  \prod_{i=2}^nx_i^d\cdot \omega  
		\quad \text{  and  } \quad
			\prod_{j=1}^m (D-\beta_j)\omega  =  \prod_{j=1}^my_j^d\cdot \omega  .
		$$
	Then, we deduce from \eqref{eq::Gauss-Manin} that
		$$
			\prod_{i=1}^n (D-\alpha_i)\omega  =D\Big( \prod_{i=2}^nx_i^d\cdot \omega  \Big)=z\prod_{j=1}^my_j^d\cdot \omega  = z\prod_{j=1}^m (D-\beta_j)\omega  .
		$$
	Using Lemma~\ref{lem::computation-GaussManin}, we deduce that $\omega\neq 0$. So we get a nonzero morphism 
		\begin{equation}\label{eq:cyclic-vector}
			\begin{split}
				\scr{D}_{S}/\rr{Hyp}(\alpha;\beta)\to \oplus_{i=0}^{n-1} \mathcal{O}_{\bb{G}_m}\cdot D^i\omega\subset \mathcal{H}^{n+m-1}_{\mathrm{dR}}(U/S,f)^{G,\tilde\chi\times \rho}
			\end{split}
		\end{equation} 
	defined by sending $1$ to $\omega $. 
	Since the left-hand side is irreducible, and both sides have the same ranks, 
	the above morphism is an isomorphism. 
	By Proposition \ref{prop::geometric-realization-convolution}, $\omega $ is a cyclic vector of $\Hyp(\alpha;\beta)
	\simeq \mathcal{H}^{n+m-1}_{\mathrm{dR}}(U/S,f)^{G,\tilde\chi\times \rho}$.
\end{proof}

\begin{rem}\label{rem::dual-hypergeometric} 
If we replace $\mathcal{E}^{f}$ by $({\bb{G}_{m}},\mathrm{d}-\mathrm{d}f)=({\bb{G}_{m}},\mathrm{d}+\mathrm{d}f)^{\vee}$, the direct sum $\bigoplus_{i=0}^{n-1}\cc{O}D^i\omega $ is the $(G,\tilde \chi\times \rho)$-isotypic component of
	$\mathcal{H}^{n+m-1}_{\mathrm{dR}}(\bb{G}_m^{n+m}/\bb{G}_{m},-f),$
	isomorphic to the connection $\Hyp_{(-1)^{n-m}}(  \alpha ;  \beta )$. To see this, it suffices to notice that the corresponding identities in Lemma~\ref{lem::computation-GaussManin} become
		$$(D-t_i/d)\omega_{t,s}=-x_i^d\omega_{t,s} \text{ and } (D+s_j/d)\omega_{t,s}=-y_j^d\omega_{t,s}$$
	in this case. The rest of the proof relies on the same calculation above and Remark~\ref{rem::dual-mod-hypergeometric}.
\end{rem}

\begin{secnumber} \textbf{Resonant case.}
	When $(\alpha,\beta)$ is resonant, the modified hypergeometric $\scr{D}$-module $\Hyp(*;\alpha;\beta)$ depends only on the classes of $\alpha$ and $\beta$ modulo $\mathbb{Z}$. 
	In \cite[6.3.8]{katz1990exponential}, Katz asked whether $\Hyp(*;\alpha;\beta)$ is isomorphic to the connection $\Hyp\bigl((\alpha_i+r_i);(\beta_j+s_j)\bigr)$ \eqref{eq:Hyp-conn} for suitable integers $r_i,s_j\in \bb{Z}$. 
	We provide a positive answer to this question in the following proposition. 
\end{secnumber}

\begin{prop} \label{p:compareDmods}
	When $(\alpha,\beta)$ is resonant, there exists a positive integer $h$ depending on $\alpha\mod \bb{Z}$ and $\beta\mod \bb{Z}$, such that for any integers $r,s>h$, the modified hypergeometric $\scr{D}$-module $\Hyp(*;\alpha;\beta)|_S$ is isomorphic to the hypergeometric connection $\Hyp\bigl((\alpha_1, \alpha_2-r,\ldots,\alpha_n-r);\beta+s\bigr)$.
\end{prop}
\begin{proof}
	We may assume that $\alpha_1=0$. Let $\tilde \omega_1,\ldots,\tilde \omega_n$ be a representative of a basis of the connection $\mathcal{H}^{n+m-1}_{\mathrm{dR}}(U/S,f)^{G,\tilde\chi\times \rho}$. More precisely, we can write 
		$$\tilde \omega_k=\sum_{e\in \bb{Z}^{n-1},f\in \bb{Z}^m} \epsilon_{k,e,f}\prod_{i=2}^n x_i^{a_i+d\cdot e_i} \prod_{j=1}^my_j^{-b_j+d\cdot f_j}\frac{\mathrm{d}x_2}{x_2}\cdots\frac{\mathrm{d}x_{n}}{x_{n}} \frac{\mathrm{d}y_1}{y_1}\cdots \frac{\mathrm{d}y_m}{y_m},$$
	where only finitely many $\epsilon_{k,e,f}$ are non-zero. We equip $\bb{Z}^{n+m-1}$ with the partial order defined by the relation that $ a\geq b $ if $ a-b\in \bb{N}^{n+m-1}$. Let $(e_0,f_0)$ be a maximal element in the set $\{(e',f')\mid (e',f')\leq (e,f) \text{ if }\epsilon_{k,e,f}\neq 0\}$. Then we take $h$ to be the the maximal value among $\{|(e_0)|_i, |(f_0)|_j\}$.  
	
	For any $r,s>h$, as in \cref{prop::geometric-realization}, we define a morphism of $\mathscr{D}$-modules:
	\begin{equation}
		\begin{split}
			\scr{D}_{S}/\rr{Hyp}(0,\alpha_2-r,\dots,\alpha_n-r;\beta+s)\to 
			\oplus_{i=0}^{n-1} \mathcal{O}_{\bb{G}_m}\cdot D^i\omega\subset \mathcal{H}^{n+m-1}_{\mathrm{dR}}(U/S,f)^{G,\tilde\chi\times \rho}
			\end{split}
		\label{eq:cyclic-vector-2}
	\end{equation}
	by sending $1$ to 
	\begin{equation*}
		\begin{split}
			\omega=\prod_{i=2}^n x_i^{a_i-d\cdot r}\cdot \prod_{j=1}^m y_j^{-b_j-d\cdot s} \frac{\mathrm{d}x_2}{x_2}\cdots\frac{\mathrm{d}x_n}{x_n} \frac{\mathrm{d}y_1}{y_1}\cdots \frac{\mathrm{d}y_m}{y_m}.
		\end{split}
	\end{equation*}
	Since for all $(e,f)$ with $\epsilon_{k,e,f}\neq 0$, we have $a_i+d\cdot e_i\geq a_i-d\cdot r_i$ and $b_j+d\cdot f_j\geq b_j-d\cdot s_j$ for any $i$ and $j$, we deduce that the class defined by $\prod_{i=2}^n x_i^{a_i+d\cdot e_i} \prod_{j=1}^my_j^{-b_j+d\cdot f_j}\frac{\mathrm{d}x_2}{x_2}\cdots\frac{\mathrm{d}x_{n}}{x_{n}} \frac{\mathrm{d}y_1}{y_1}\cdots \frac{\mathrm{d}y_m}{y_m}$ lies in the image of \eqref{eq:cyclic-vector-2} by \cref{lem::computation-GaussManin}. This morphism is surjective and therefore, an isomorphism.
\end{proof}

\section{Irregular Hodge filtration of hypergeometric connections}\label{sec::hodge-number}

This section aims to calculate the (irregular) Hodge filtrations of hypergeometric connections (see \cref{thm::hodge-number} and \cref{thm:Hodge-fil}). 

In this section, let $n\geq m \ge 0$ be two integers, $\alpha=(\alpha_1,\ldots,\alpha_n)$ and $\beta=(\beta_1,\ldots,\beta_j)$ two sequences of non-decreasing rational numbers in $[0,1)$. 

\subsection{Exponential mixed Hodge structures}\label{subsec::emhs}
To explain certain duality on the irregular Hodge filtration of hypergeometric connections, we use the language of exponential mixed Hodge structures introduced by Kontsevich-Soibelman \cite{Kontsevich2011}. We recall the basic definitions of exponential mixed Hodge structures from \cite[Appx.]{fresan2018hodge}.

Let $X$ be a smooth algebraic variety and $K$ a number field. We denote by $\mathrm{MHM}(X, K)$ the abelian category of \textit{mixed Hodge modules} on $X$ with coefficients in $K$. 
In particular, when $X=\mathrm{Spec}(\mathbb{C})$, the category $\mathrm{MHM}(X,K)$ is equivalent to the category of mixed $K$-Hodge structures. 
The bounded derived categories $\mathrm{D}^b(\mathrm{MHM}(X, K))$ admit a six functor formalism. 
We put a lower left subscript $_{\rr{H}}$ for the functors of mixed Hodge modules. For more details about mixed Hodge modules, see \cite{saito1990}.

Let $\pi\colon \bb{A}^1\to \mathrm{Spec}(\bb{C})$ be the structure morphism. The category $\rm EMHS(K)$ of \textit{exponential mixed Hodge structures} with coefficients in $K$ is defined as the full subcategory of $\mathrm{MHM}(\bb{A}^1,K)$, whose objects $N^{\mathrm{H}}$ have vanishing cohomology on $\bb{A}^1$, i.e., satisfying $_{\mathrm{H}}\pi_*N^{\mathrm{H}}=0$.

There is an exact functor $\Pi\colon \mathrm{MHM}(\bb{A}^1,K)\to \mathrm{MHM}(\bb{A}^1,K)$ defined by
	\begin{equation}\label{eq::projector}
	N^{\mathrm{H}} \mapsto ~_{\mathrm{H}}s_*(N^{\mathrm{H}}\boxtimes _{\mathrm{H}}j_!\mathcal{O}_{\bb{G}_m}^{\mathrm{H}})
	\end{equation}
where $j\colon \bb{G}_{m,\bb{C}}\to \bb{A}^1$ is the inclusion and $s\colon\bb{A}^1\times \bb{A}^1\to \bb{A}^1$ is the summation map. The functor $\Pi$ is a projector onto $\rm{EMHS}(K)$, i.e. it factors through $\rm{EMHS}(K)$ with essential image $\rm{EMHS}(K)$. Using this functor, the dual of an object $M$ in $\rm{EMHS}(K)$ is defined by $\Pi( ~_{\rr{H}}[t\mapsto -t]^* \bb{D}(M))$, where $t$ is the coordinate of $\bb{A}^1$.

For each object $\Pi(N^{\mathrm{H}})$ of the category $\rm{EMHS}(K)$, there exists a weight filtration $W^{\mathrm{EMHS}}_\bullet$ on $\Pi(N^{\mathrm{H}})$, defined by the weight filtration on $N^{\mathrm{H}}$: 
	$W_n^{\rr{EMHS}}\Pi (N^{\mathrm{H}}):=\Pi(W_n N^{\mathrm{H}})$. 
	We will drop the superscript for simplicity.

	The \textit{de Rham fiber} functor from $\rm{EMHS}(K)$ to $\mathrm{Vect}_{\bb{C}}$ is defined by 
\begin{equation}\label{eq::de Rham fiber}
	\Pi(N^{\mathrm{H}})\mapsto\Hdr{1}(\bb{A}^1,\Pi(N)\otimes \mathcal{E}^{t}),
\end{equation}
where $\Pi(N)$ is the underlying $\mathscr{D}$-module of $\Pi(N^{\rr{H}})$ and $\mathcal{E}^t$ denotes the exponential $\mathscr{D}$-module $(\mathcal{O}_{\bb{A}^1},\mathrm{d}+\mathrm{d}t)$.  

The de Rham fiber functor is faithful and one can associate an \textit{irregular Hodge filtration} $F^{\bullet}_{\rr{irr}}$ on the de Rham fibers of objects in $\rm{EMHS}(K)$ by \cite[\S6.b]{Sabbah-2010-laplace}, compatible with the definitions in \cite{esnault17e1,Sabbah-2010-laplace,Sabbah-Yu-2015}. 

\subsubsection{Objects of EMHS attached to regular functions} \label{sss:EMHS f}
Let $X$ be a smooth affine variety of dimension $n$ and $K$ a number field. We denote by $K^{\rr{H}}_X$ the trivial Hodge module on $X$ with coefficients in $K$. For a regular function $g\colon X\to \bb{A}^1$ and an integer $r$, we consider the following exponential mixed Hodge structures 
\begin{equation*}
	\mathrm{H}^r(X,g):=\Pi(\mathcal{H}^{r-n}~_{\mathrm{H}}g_*K_X^{\rr{H}}),\ \mathrm{H}^r_c(X,g):=\Pi(\mathcal{H}^{r-n}~_{\mathrm{H}}g_!K_X^{\rr{H}}).
\end{equation*}
The exponential mixed Hodge structures $\rr{H}^r(X,g)$, $\rr{H}_c^r(X,g)$ are mixed of weights at least $r$ and mixed of weights at most $r$ respectively by {\cite[A.19]{fresan2018hodge}}.
 
The de Rham fiber of $\mathrm{H}^r_{?}(X,g)$ is isomorphic to $\mathrm{H}^r_{\mathrm{dR},?}(X,g)$, and the irregular Hodge filtration on the de Rham fibers are identified with those on twisted de Rham cohomologies introduced in \cite{yu2012irregular}.

\subsubsection{The irregular Hodge filtration on twisted de Rham cohomology}
	We briefly recall the definition of the irregular Hodge filtration on the twisted de Rham cohomology following \cite{yu2012irregular}. Let $X$ and $g$ be as above, $j\colon X\to \bar X$ a smooth compactification of $X$, and $D:=\bar X\backslash X$ the boundary divisor. The pair $(\bar X,D)$ is called a \textit{good compactification} of the pair $(X,g)$, if $D$ is normal crossing and $g$ extends to a morphism $\bar g\colon \bar X\to \bb{P}^1$.

Let $P$ be the pole divisor of $g$. The twisted de Rham complex $(\Omega_{\bar X}^\bullet(*D),\nabla=\mathrm{d}+\mathrm{d}g)$ admits a decreasing filtration $F^\lambda(\nabla):=F^0(\lambda)^{\geq \lceil\lambda\rceil}$, indexed by non-negative real numbers $\lambda$, where $F^0(\lambda)$ is the complex
	$$\mathcal{O}_{\bar X}(\lfloor -\lambda P\rfloor)\xrightarrow{\nabla} \Omega_{\bar X}^1(\log D)(\lfloor (1-\lambda)P\rfloor)\to \cdots \to \Omega_{\bar X}^p(\log D)(\lfloor (p-\lambda)P \rfloor)\to \cdots.$$

The \textit{irregular Hodge filtration} on the de Rham cohomology $\mathrm{H}^1_{\dR}(X,g)$ is defined by
	\begin{equation}\label{eq::irreg-hodge-fil} 
		\begin{split}
			F_{\rr{irr}}^\lambda\mathrm{H}^i_{\mathrm{dR}}(X,g):=\im(\bb{H}^i(\bar X,F^\lambda(\nabla))\to \mathrm{H}^i_{\mathrm{dR}}(X,g)),
		\end{split}
	\end{equation}
which is independent of the choice of the good compactification $(\bar X,D)$ \cite[Thm.\,1.7]{yu2012irregular}.

When $X$ is isomorphic to a torus $\mathbb{G}_m^n$, the regular function $g$ on $X$ is a Laurent polynomial of the form $\sum_{P=(p_1,\ldots,p_n)}c(P)x^P$. We refine the normal fan of the Newton polytope $\Delta(g)$ to make the associated toric variety $X_{\rr{tor}}$ smooth proper. Although $(X_{\rr{tor}},D_{\rr{tor}}=X_{\rr{tor}}\backslash X)$ is not a good compactification for the pair $(X,g)$ in general, we can still define $F^{\lambda}_{\mathrm{NP}}(\nabla)$ and the \textit{Newton polyhedron filtration} $F^\lambda_{\mathrm{NP}}\mathrm{H}^1_{\dR}(U,\nabla)$ similarly to that in \eqref{eq::irreg-hodge-fil} by replacing the good compactification $(\bar X,D)$ with $(X_{\tor},D_{\tor})$,

When $g$ is non-degenerate with respect to $\Delta(g)$, by \cite[Thm 1.4]{adolphson1997twisted} and \cite[Thm 4.6]{yu2012irregular}, the only non-vanishing twisted de Rham cohomology group of the pair $(X,g)$ is the middle cohomology group $\rr{H}^n_{\dR}(X,g)$, and the irregular Hodge filtration $F_{\rr{irr}}^\bullet$ agrees with the Newton polyhedron filtration $F^\bullet_{\rr{NP}}$ on $\Hdr{n}(X,g)$. In particular, we have
	$$\bb{H}^i(X_{\tor},F^\lambda_{\mathrm{NP}}(\nabla))=\mathrm{H}^i(\Gamma(X_{\tor},F^\lambda_{\mathrm{NP}}(\nabla))),$$
which allows us to compute the irregular Hodge filtration via the knowledge of $\Delta(g)$. 

Now, we present an explicit way to calculate the Newton polyhedron filtration. 
For a cohomology class $\omega= x^Q\frac{\rr{d}x_1}{x_1}\wedge\cdots\wedge \frac{\rr{d}x_n}{x_n}$ such that the lattice point $Q=(q_1,\ldots,q_n)$ lies in $\mathbb{R}_{\geq 0}\Delta(g)$, we define $w(Q)$ to be the weight of $Q$ in the sense of \cite{adolphson1997twisted}, i.e. the minimal positive real number $w$ such that $Q\in w\cdot \Delta(g)$. The associated cohomology class of $\omega$ lies in $F^\lambda_{\rr{NP}}\Hdr{n}(X,g)$ if
	$$\omega \in \Gamma(X_{\tor},\Omega_{X_{\tor}}^n(\log D_{\tor})(\lfloor (n-\lambda)P \rfloor)).$$
	Notice that each ray $\rho$ in the normal fan of $\Delta(g)$ corresponds to an irreducible component $P_{\rho}$ of $P$. Let $v_{\rho}$ be a primitive vector of the ray $\rho$. Then, the multiplicity of $\omega$ along $P_\rho$ is given by $<Q,v_{\rho}>$ \cite[p.\,61]{Fulton93}. Taking the multiplicities of $P_{\rho}$ in $P$ into account, we have
\begin{equation}\label{eq:newton-monomial}
	x^Q\frac{\rr{d}x_1}{x_1}\wedge\cdots\wedge \frac{\rr{d}x_n}{x_n}\in F^{n-w(Q)}_{\mathrm{NP}}\mathrm{H}^n_{\dR}(X,g),
\end{equation} 
as indicated in \cite[p.\,126 footnote]{yu2012irregular}.

\subsubsection{The EMHS associated with hypergeometric connections}\label{sec:emhs-hyp}

In this subsection, we assume $\alpha_1=0$ and let $\tilde \chi\times \rho$ be the product of characters associated with $\alpha_i$ and $\beta_j$ in \eqref{eq:char}. 
 
\begin{defn}\label{defn::hypergeometric-EMHS}
	Let $K$ be the number field $\bb{Q}(\zeta_d^{a_i},\zeta_d^{b_j})$ and $a \in S(\bb{C})$. For $?\in \{\emptyset,c\}$, we define 
		\begin{equation*}
			{E}_?(a;\alpha;\beta):=\mathrm{H}^{n+m-1}(\bb{G}_m^{n+m-1},f_a)^{\left(G, \tilde\chi\times \rho\right)}
		\end{equation*}
	as exponential mixed Hodge structures with coefficients in $K$ in the sense of \eqref{sss:EMHS f}.
\end{defn}

By Proposition~\ref{prop::geometric-realization-convolution} and the base change theorem, the de Rham fiber of $ {E}(a;\alpha;\beta)$ is isomorphic to the fiber of $\Hyp_\lambda(\alpha;\beta)$ at $a\cdot \lambda\in S(\bb{C})$, for $\lambda\in \bb{Q}\cros$. 

Let $t$ be the largest natural number such that $\alpha_t=0$. We let $\bar\alpha$ and $\bar \beta$ be the sequences of rational numbers defined by
	\begin{equation}\label{eq::conjugate-sequence-alpha}	\bar\alpha_i=
		\begin{cases}
			0 &1\leq k\leq t,\\ 
			1-\alpha_{n+t+1-k} & t+1\leq k\leq n,
		\end{cases}
		\quad \textnormal{and} \quad \bar\beta_k=1-\beta_k. 
	\end{equation}

\begin{prop}\label{prop::EMHS-properties}

\begin{enumerate}[label=(\arabic*).]
	\item The dual of the exponential mixed Hodge structure $ {E}_c(a;\alpha;\beta)$ is isomorphic to $ {E}((-1)^{n-m}a;\bar \alpha;\bar \beta)$.
	\item 
	When $(\alpha, \beta)$ is non-resonant, the exponential mixed Hodge structures $ {E}_?(a;\alpha;\beta)$ for $?\in \{\varnothing, c\}$ are all isomorphic. In particular, they are pure of weight $n+m-1$.
	\end{enumerate}
\end{prop}
\begin{proof}
	(1) The EMHS $\mathrm{H}^{n+m-1}(\bb{G}_m^{n+m-1},f_a)$ is dual to $\mathrm{H}^{n+m-1}(\bb{G}_m^{n+m-1},-f_a)$, which is also isomorphic to $\mathrm{H}^{n+m-1}(\bb{G}_m^{n+m-1},f_{(-1)^{n-m}a})$. By taking the corresponding isotypic components of them, we deduce the first assertion.

	(2)  Since the de Rham fiber functor is faithful, the forget support morphism
		$${E}_c(a;\alpha;\beta) \to {E}(a;\alpha;\beta)$$
	is an isomorphism by Corollary~\ref{cor::forget-support-iso}. 
	Hence, the exponential mixed Hodge structures ${E}_c(a;\alpha;\beta)$ and ${E}(a;\alpha;\beta)$ are isomorphic, and are pure of weight $n+m-1$.
\end{proof}

\begin{rem}
We can define confluent hypergeometric  motives as exponential motives in the sense of \cite{fresan2018exponential}, such that the exponential mixed Hodge structures $E(1;\alpha;\beta)$ and the hypergeometric connections $\Hyp(\alpha,\beta)$ are their Hodge realizations and $\scr{D}$-module realizations respectively. 

More precisely, assume $n>m$ and let $\zeta_{n-m}$ be an $(n-m)$-th primitive root of unity.  The group $\mu_d^{n+m-1}$ acts on $\bb{G}_{m,(x_i,y_j)}^{n+m-1}\times \bb{G}_{m,t}$ as before on the coordinates $(x_i,y_j)$, and the group $\mu_{n-m}$ acts on the coordinates $(x_i,y_j,t)$ by 
			$$\zeta_{n-m}\cdot (x_i,y_j,t)=(\zeta_{n-m}\inv x_i,\zeta_{n-m}\inv y_j,\zeta_{n-m}t).$$
Then we define the confluent hypergeometric motives as 
    $$\mathrm{H}^{n+m-1}(\bb{G}_m^{n+m-1},f_1)^{\left(\mu_{d}^{n+m-1}\times \mu_{n-m}, \tilde\chi\times \rho\times 1\right)},$$
which is a priori defined over $K(\zeta_{n-m})$ with coefficients in $K$. 

Here, the field of definition $K(\zeta_{n-m})$ is not optimal. For example, using an argument similar to that in \cite[Rem.\,3.3]{qin23function} and the Galois descent \cite[Thm.\,5.2.4]{fresan2018exponential}, one can show that these motives are defined over $K$. If $L$ is a subfield of $K$ such that $\gal(\bb{C}/L)$ preserves both the sets $\{\exp(2\pi i \alpha_i)\}$ and $\{\exp(2\pi i \beta_j)\}$, then these motives can further descend to $L$ (see also \cite[Thm.\,1.1]{Betti-hypergeometric-} for a related discussion).
\end{rem}

\subsection{A basis in relative twisted de Rham cohomology}\label{subsec::basis}
In this subsection, we assume $\alpha_1=0$. 
We define positive integers $s_1,\ldots,s_{m+1}$ by
	$$
		s_r=
		\begin{cases}
			1 & r=0\\
			\#\{i\colon \alpha_{i}<\beta_r\} & 1\leq r\leq m\\
			n+1 &r=m+1
		\end{cases}
	$$
	and for $r$ and $\ell$ such that $0\leq r\leq m$ and $1\leq \ell\leq s_{r+1}-s_r$
	, we set 
$$
		g_{r,\ell}=
		x_2^{a_2}\cdots x_{s_r+\ell-1}^{a_{s_r+\ell-1}}\cdot x_{s_r+\ell}^{a_{s_r+\ell}-d}\cdots  x_{n}^{a_{n}-d} \cdot y_1^{d-b_1}\cdots y_r^{d-b_r}\cdot y_{r+1}^{2d-b_{r+1}}\cdots y_m^{2d-b_m},
	$$
Let $\eta=\frac{\mathrm{d}x_2}{x_2}\cdots\frac{\mathrm{d}x_n}{x_n} \frac{\mathrm{d}y_1}{y_1}\cdots \frac{\mathrm{d}y_m}{y_m}$ and $\omega_{r, \ell}=g_{r,\ell}\cdot \eta$ be the corresponding differential forms in $\mathcal{H}_{\dR}^{n+m-1}(U/S,\pm f)^{(G,\tilde\chi\times \rho)}$, where $U$ and $S$ are defined in \eqref{eq:diagram-Hyp-mot}.

\begin{prop} \label{p:coh basis}
	If $(\alpha, \beta)$ is non-resonant, the cohomology classes defined by 
	\[
		\omega_{r,\ell},\quad 0\leq r\leq m,\quad 1\leq \ell\leq s_{r+1}-s_{r}
	\]
	in $\mathcal{H}_{\dR}^{n+m-1}(U/S,\pm f)^{(G,\tilde\chi\times \rho)}$ form a basis over $\mathcal{O}_S$.
\end{prop}

\begin{proof}
	It suffices to show that $\mathrm{span}(\omega_{r,\ell})=\mathrm{span}(D^i\omega\mid 0\leq i\leq n-1)$ for a cyclic vector $\omega$. 

	To a Laurent monomial $g=\prod x_i^{u_i} \prod y_j^{v_j}$ in variables $\{x_i\}_{i=2}^n$ and $\{y_j\}_{j=1}^m$, 
	we associate a lattice point $\point(g)=(u_2,\dots,u_n,v_1,\dots,v_m)\in \bb{Z}^{n+m-1}\subset \bb{R}^{n+m-1}$.
	If $\omega=g\cdot \eta$ is the product of a monomial $g$ with the differential form $\eta$, we set $\point(\omega):=\point(g)$ for the corresponding point. 
	
	Let $\pi_1$ and $\pi_2$ be the projections from $\bb{R}^{n+m-1}$ to $\bb{R}^{n-1}_{u_i}$ and $\bb{R}^m_{v_j}$ respectively. The for the differential forms $\omega_{r,\ell}$, we have  
	$$\pi_1(\point(\omega_{r,\ell}))=(a_2,\ldots,a_{s_r+\ell-1},a_{s_r+\ell}-d,\ldots,a_n-d)$$
and
	$$\pi_2(\point(\omega_{r,\ell}))=(d-\beta_1,\ldots,d-\beta_{r},2d-\beta_{r+1},\ldots,2d-\beta_m).$$
	It follows from Lemmas~\ref{lem::Newton-Polygon-ineq-m=0}, \ref{lem::Newton-Polygon-ineq-m-neq-0}, and \ref{lem::Newton-Polygon-ineq-n=m} that the fibers of all $\omega_{r,\ell}$ at $a\in S(\bb{C})$ lie in the cone $\bb{R}_{\geq 0}\cdot \Delta(f_a)$ 

Let $P_i$ and $Q_j$ be the points corresponding to monomials $x_i^d$ and $y_j^d$ respectively for $2\leq i\leq n$ and $1\leq j\leq m$.

\begin{lem}\label{lem::trick}
For a point $P\in \mathbb{Z}^{n+m-1}$ and two integers $2\leq i_0\leq n$ and $1\leq j_0\leq m$, let $\omega_0, \omega_1$ and $\omega_2$ be the corresponding differential forms of the points $P, P+Q_{j_0}$ and $P+P_{i_0}$ in $\mathbb{Z}^{n+m-1}$. If the $i_0$-th coordinate of $P$ is different from the negative of the $j_0$-th coordinate of $P$, then we have 
$$\mathrm{span}(\omega_0,\omega_2)=\mathrm{span}(\omega_1,\omega_2)\quad \textnormal{in } ~ \rr{H}^{n+m-1}_{\dR}(U/S,f)^{(G,\tilde\chi\times \rho)}.$$
\end{lem}
\begin{proof}
	Let $P$ be the point $(t_i,s_j)\in \mathbb{Z}^{n+m-1}$ and $\omega_0$ be the associated differential form. By the assumption, we have $t_i\neq - s_j$. Then
		$$\mathrm{span}(\omega_0,(D-t_{i_0})\omega_0)=\mathrm{span}((D-t_{i_0})\omega_0,(D+s_{j_0})\omega_0).$$
	At last, notice that we have $\omega_1=(D+s_{j_0})\omega_0$ and $\omega_2=(D-t_{i_0})\omega_0$ by Lemma~\ref{lem::computation-GaussManin} and Remark~\ref{rem::dual-hypergeometric}.
\end{proof}

\noindent \textbf{Step 1}: For $1\leq \ell\leq s_1-s_0$, we replace the differential forms $\omega_{0,\ell}$ by differential forms $\omega_{0,\ell}^{(1)}$ of the form $g\cdot \eta$, such that
	 	$$\point(\omega_{0,\ell}^{(1)})=\point(\omega_{0,\ell})-Q_1.$$
	In other words, we keep the first $n-1$ coordinates unchanged and change the last $m$ coordinates to
		$$(d-\beta_1,2d-\beta_{2},\ldots,2d-\beta_m).$$
	We also put $\omega_{r,\ell}^{(1)}=\omega_{r,\ell}$ for $r\geq 1$. Then use Lemma~\ref{lem::trick}, we check that 
		$$\operatorname{Span}(\omega_{r,\ell}^{(1)}\mid r,\ell)=\operatorname{Span}(\omega_{r,\ell}\mid r,\ell).$$
\noindent\textbf{Step $i\geq2$}: Assume that we have already obtained elements $\omega_{r,\ell}^{(i-1)}$ for $i\geq 2$. Let $\omega_{r,\ell}^{(i)}$ be differential forms of the form $g\cdot \eta$, such that 
		$$
			\point(\omega_{r,\ell}^{(i)})=
			\begin{cases}
				\point(\omega_{r,\ell}^{(i-1)})-Q_i & \textnormal{ if } r\leq i-1,\\
				\point(\omega_{r,\ell}^{(i-1)}) & \textnormal{ if }  i\leq r\leq m,
			\end{cases}
		$$
	i.e., when $r\leq i-1$, we keep the first $n-1$ coordinates of $\point(\omega_{r,\ell}^{(i-1)})$ unchanged, and replace the last $m$ coordinates of $\point(\omega_{r,\ell}^{(i-1)})$ by
		$$(d-\beta_1,\ldots, d-\beta_{i},2d-\beta_{i+1},\ldots,2d-\beta_m).$$
	Then by Lemma~\ref{lem::trick}, we check that 
		$$\operatorname{Span}(\omega_{r,\ell}^{(i)}\mid r,\ell)=\operatorname{Span}(\omega_{r,\ell}^{(i-1)}\mid r,\ell)=\operatorname{Span}(\omega_{r,\ell}\mid r,\ell).$$
	
\noindent\textbf{After Step $m$}: After $m$ steps, we get $\omega_{r,\ell}^{(m)}$ such that 
	$$
		\point(\omega_{r,\ell}^{(m)})=(a_2,\ldots,a_{s_r+\ell-1},a_{s_r+\ell}-d,\ldots,a_n-d,d-\beta_1,\ldots,d-\beta_m).
	$$
Note that the there is a bijection between $\{(s_r,\ell)\}_{0\le r \le m, 1\le \ell\le s_{r+1}-s_r}$ and $\{1,\cdots,n\}$ by sending $(r,\ell)$ to $s_r+\ell-1$. We set $\tilde \omega_{s_r+\ell-1}=\omega_{r,\ell}^{(m)}$ via this map. Then by Lemma~\ref{lem::computation-GaussManin}, we have $\tilde \omega_{i+1}=(D-1+\frac{a_{i+1}}{d})\tilde \omega_i$ for $1\leq i\leq n-1$. It follows that 	
\begin{eqnarray*}	
		\operatorname{Span}(D^i\tilde \omega_1\mid 0\leq i\leq n-1)
		&=& \operatorname{Span}(\tilde \omega_{i}\mid 1\leq i\leq n) \\
		&=& \operatorname{Span}(\omega_{r,\ell}^{(m)}\mid r,\ell) \\
		&=& \operatorname{Span}(\omega_{r,\ell}\mid r,\ell).
	\end{eqnarray*}
By Proposition \ref{prop::geometric-realization} and Remark \ref{r:cyclic-vector}, $\widetilde{\omega}_1$ is a cyclic vector, from which we showed that $\{\omega_{r,\ell}\}_{r,\ell}$ form a basis. This finishes the proof. 
\end{proof}

\subsection{Calculation of the irregular Hodge filtration}

Recall that a non-resonant hypergeometric connection is rigid. Hence, it underlies an irregular mixed Hodge module on $\bb{P}^1$ \cite[Thm.\,0.7]{Sabbah18irregular}, and therefore, admits a unique irregular Hodge filtration $F_{\irr}^{\bullet}$ (up to a shift). 
When $n=m$, the irregular mixed Hodge module structure coincides with the variation of Hodge structures on $\Hyp(\alpha,\beta)$.

\begin{thm}\label{thm:Hodge-fil}
	Assume $(\alpha,\beta)$ is non-resonant. 
	\begin{enumerate}[label=(\arabic*).]
		\item Via the isomorphism $ \Hyp(\alpha,\beta)\simeq \mathcal{H}^{n+m-1}_{\mathrm{dR}}(U/S,f)^{(G,\tilde\chi\times \rho)}$, the irregular Hodge filtration on $\Hyp(\alpha,\beta)$ can be identified with the following filtration of sub-bundles:
			\[
				F^p_{\rr{irr}}\mathcal{H}^{n+m-1}_{\mathrm{dR}}(U/S,f)^{(G,\tilde\chi\times \rho)}
				=\bigoplus_{n+m-1-w(\omega_{r,s})\geq p}\omega_{r,s}\cc{O}_S.
			\]

		\item Up to an $\bb{R}$-shift, the jumps of the irregular Hodge filtration on $\Hyp(\alpha,\beta)$ occur at $\theta(k)$ (Theorem~\ref{thm::hodge-number}) and for any $p\in \bb{R}$ we have
		$$\rr{rk}\,\mathrm{gr}^p_{F_{\irr}}\Hyp(\alpha;\beta) = \#\theta\inv(p).$$
	\item The irregular Hodge filtration satisfies the Griffiths' transversality \footnote{A general statement about the Griffiths' transversality on irregular Hodge filtration is discussed in \cite[Rem.\,6.3]{Sabbah-Yu-2015}.}, that is, 
		$$\nabla (F^p_{\rr{irr}} \Hyp(\alpha,\beta)) \subset \Omega_S^1\otimes F^{p-1}_{\rr{irr}} \Hyp(\alpha,\beta),\quad \forall p\in \mathbb{R}.$$
	\end{enumerate}
\end{thm}

\begin{rem}
	Inspired by the Griffiths' transversality, we expect that there exist oper structures on the hypergeometric connections, which refine the irregular Hodge filtrations. 
	An oper structure plays an essential role in the geometric Langlands correspondence \cite{BD,Zhu,KXY}. 
\end{rem}

To prove the above theorem, we study the Hodge numbers of the irregular Hodge filtration on fibers as explained in \cref{s:idea}. 

\begin{thm}\label{thm::hodge-number}
	Up to an $\bb{R}$-shift, the jumps of the irregular Hodge filtration $F^\bullet_{\rr{irr}}$ on the fiber $\Hyp(\alpha;\beta)_a$ occur at
		$$\theta(k):= (n-m)\alpha_k+\#\{i\mid \beta_i< \alpha_k\}+(n-k)-\sum_{i=1}^n \alpha_i+\sum_{j=1}^m\beta_j$$
	for $1\leq k\leq n$. Moreover, we have
		$\operatorname{dim}\mathrm{gr}^p_{F_{\irr}}\Hyp(\alpha;\beta)_a = \#\theta\inv(p)$ for any $p\in \bb{R}$.
\end{thm}

\subsubsection{Proof of Theorem \ref{thm::hodge-number}}
\begin{proof}	We may assume $\alpha_1=0$ by \eqref{eq:reduce-to-alpha1=0}. By Proposition~\ref{prop::geometric-realization-convolution} and Definition~\ref{defn::hypergeometric-EMHS}, we have 
\begin{eqnarray} \label{eq:filtrations fiber}
			F_{\irr}^\bullet \Hyp(\alpha;\beta)_a &\simeq & F_{\irr}^\bullet\mathrm{H}^{n+m-1}_{\mathrm{dR}}(\bb{G}_m^{n+m-1},f_a)^{(G,\tilde\chi\times \rho)}\\
			&\simeq&  F_{\irr}^\bullet\mathrm{H}^{n+m-1}_{\mathrm{dR}}(\bb{G}_m^{n+m-1},-f_{(-1)^{n-m}a})^{(G,\tilde\chi\times \rho)},\nonumber
	\end{eqnarray}
	where $\tilde \chi$ and $\rho$ are products of characters corresponding to $\alpha_i$ and $\beta_j$ from \eqref{eq:char}. So it suffices to compute the irregular Hodge filtration on the twisted de Rham cohomologies $\mathrm{H}^{n+m-1}_{\mathrm{dR}}(\bb{G}_m^{n+m},\pm f_{ a})^{(G,\tilde\chi\times \rho)}$. Since $f_a$ is non-degenerate with respect to $\Delta(f_a)$, we can compute the filtration in terms of Newton polyhedron filtration. 

	Let $\omega_{r,\ell}$ be the basis of $\Hyp(\alpha;\beta)_a$ from Proposition \ref{p:coh basis}. Recall that $w(\omega_{r,\ell})$ is the minimal positive real number $w$ such that $\point(g_{r,\ell})\in w\cdot \Delta(f_a)$. It follows from \eqref{eq:newton-monomial} that
		$$\omega_{r,\ell} \in F_{\irr}^{n+m-1-w(\omega_{r,\ell})} \mathrm{H}^{n+m-1}_{\mathrm{dR}}(\bb{G}_m^{n+m-1},\pm f_a)^{(G,\tilde\chi\times \rho)}.$$

	We consider an auxiliary filtration 	
		$G^\bullet$ on $\HH^{n+m-1}_{\mathrm{dR}}(\bb{G}_m^{n+m-1},\pm f_a)^{(G,\tilde\chi\times \rho)}$ defined by
		\begin{equation}\label{eq:auxiliary-fil}
			G^p:=\operatorname{span}\{\omega_{r,\ell}\mid n+m-1- w(\omega_{r,\ell})\geq p\}.
		\end{equation}
	By the following lemmas, $F^\bullet$ coincides with $G^\bullet$, which finishes the proof of the theorem. 
\end{proof}
	
	\begin{lem}\label{lem::Newton-weight}
		We set $\theta(n+1)=\theta(1)$. For $0\le r\le m$, $1\le \ell\le s_{r+1}-s_r$, we have 
	\[
		n+m-1-w(\omega_{r,\ell})=\theta(s_r+\ell).
	\]
	\end{lem}

	\begin{lem}\label{lem:dual-dimension-G}
		For $0\leq p\leq n+m-1$, we have
			$$\scalemath{0.95}{\dim \gr_G^{p}\mathrm{H}^{n+m-1}_{\mathrm{dR}}(\bb{G}_m^{n+m-1},\pm f_a)^{(G,\tilde\chi\times \rho)}
			=\dim \gr_G^{n+m-1-p}\mathrm{H}^{n+m-1}_{\mathrm{dR}}(\bb{G}_m^{n+m-1},\mp f_a)^{(G,\tilde \chi\inv\times\rho\inv)}.}$$
	\end{lem}

	\begin{lem} \label{l:two filtrations}
	The two filtrations $F^\bullet_{\irr}$ and $G^\bullet$ coincide.
	\end{lem}

	\begin{proof}[Proof of Lemma \ref{lem::Newton-weight}] 	
	By Lemmas~\ref{lem::Newton-Polygon-ineq-m=0}, \ref{lem::Newton-Polygon-ineq-m-neq-0}, and~\ref{lem::Newton-Polygon-ineq-n=m}, the weight $w(\omega_{r,\ell})$ equals to the number $\max_{k}\{h_{k}(g_{r,\ell})/d\}$, where $h_k$ are defined in \eqref{eq:h_k-1}, \eqref{eq:h_k-2}, and \eqref{eq:h_k-3}. We can check that 
		$$w(\omega_{r,\ell})=h_{s_r+\ell}(g_{r,\ell})/d,$$
	where we put $h_1=\ldots =h_n=h_{n+1}$ when $n=m$.
	Now it suffices to check that $n+m-1-w(\omega_{r,\ell})$ agrees with one of the jumps of the irregular Hodge numbers of $\Hyp( \alpha;  \beta)_a$.
		
	If $s_r+\ell=n+1$, the monomial $g_{m,n+1-s_m}$ corresponds to the point 
		$$(a_2,\ldots,a_n,d-b_1,\ldots,d-b_m).$$ 
	Then we have
		$$n+m-1-h_{n+1}(g_{m,n+1-s_m})/d= n-1 -\sum_{i=1}^n\alpha_i + \sum_{j=1}^m \beta_j=\theta(1). $$
	
	If $s_r+\ell<n+1$, we have
		$$\begin{aligned}
			&n+m-1-h_{s_r+\ell}(g_{r,\ell})/d\\
			=&n+m-1- \biggl( \sum_{i=1}^n \alpha_i - (n+1-s_r-\ell) -\sum_{j=1}^m \beta_j + (2m-r) - (n-m)(\alpha_{s_r+\ell}-1)\biggr)\\
			=& (n-m)\alpha_{s_r+\ell} +r +(n-s_r-\ell) -\sum_{i=1}^n\alpha_i+\sum_{j=1}^m\beta_j,
		\end{aligned}$$
	which is exactly $\theta(s_r+\ell)$. 
	\end{proof}

	\begin{proof}[Proof of Lemma~\ref{lem:dual-dimension-G}]
	For simplicity, we write 
	\begin{equation}\label{eq:delta_p}
		\delta_p^{\pm}(\alpha,\beta):=\dim \gr_G^p\mathrm{H}^{n+m-1}_{\mathrm{dR}}(\bb{G}_m^{n+m-1},\pm f_a)^{(G,\tilde\chi\times \rho)}.
	\end{equation}
	Recall that in \eqref{eq::conjugate-sequence-alpha}, we let $t$ be the biggest natural number such that $\alpha_t=0$. For $1\leq k\leq t$, the numbers $\alpha_k$ and $\bar\alpha_{t+1-k}$ are $0$. And for $t+1\leq k\leq n$, we have $\bar \alpha_{n-k+t+1}=1-\alpha_k$. Then
		$$\sum_{i=1}^n \alpha_i+\sum_{i=1}^n \bar \alpha_i=n-t \text{ and }\sum_{j=1}^m \beta_j+\sum_{j=1}^m \bar\beta_j=m.$$	
	Similar to the number $\theta(k)$, we let $\bar\theta(k)$ be the numbers
		$$(n-m)\bar\alpha_k+\#\{i\mid \bar\beta_i< \bar\alpha_k\}+(n-k)-\sum_{i=1}^n \bar\alpha_i+\sum_{j=1}^m\bar\beta_j,\ 1\leq k\leq n$$
	for the sequences $\bar \alpha$ and $\bar \beta$. Then for $1\leq k\leq t$, we have
		\begin{equation*}
			\begin{split}
				\theta(k)+\bar \theta(t+1-k)
				= &\biggl( n-k- \sum_{i=1}^n \alpha_i+\sum_{j=1}^m\beta_j \biggr) +\biggl(n-(t+1-k)- \sum_{i=1}^n \bar\alpha_i+\sum_{j=1}^m\bar\beta_j \biggr)\\
				=&(2n-t-1)-(n-t)+m=n+m-1.
			\end{split}
		\end{equation*}
	For $t+1\leq k\leq n$, we have 
		\begin{equation*}
			\begin{split}
				&\theta(k)+\bar \theta(n-k+t+1)\\
				= &\biggl( (n-m)\alpha_k+\#\{j\mid \beta_j<\alpha_k\}+ n-k- \sum^n_{i=1} \alpha_i+\sum_{j=1}^m\beta_j \biggr) \\
				+& \biggl( (n-m)\bar \alpha_{n-k+t+1}+\#\{j\mid \bar\beta_j<\bar\alpha_{n-k+t+1}\} +n-(n-k+t+1)- \sum_{i=1}^n \bar\alpha_i+\sum_{j=1}^m\bar\beta_j \biggr)\\
				=&(n-m)+m+(n-t-1)-(n-t)+m=n+m-1.
			\end{split}
		\end{equation*}
	So there exists a permutation $\sigma\in S_{n}$ such that $\theta(k)+\bar\theta(\sigma(k))=n+m-1$. It follows that
	\begin{equation*}
		\begin{split}
			\delta_p^{\pm}(\alpha,\beta)
			=&\#\{k\mid \theta(k)= p\}
			=\#\{k\mid n+m-1-p= n+m-1-\theta(k)\}\\
			=&\#\{k\mid \bar \theta(k)= n+m-1-p\}
			=\delta_{n+m-1-p}^{\mp}(\bar\alpha,\bar\beta).
		\end{split}
	\end{equation*}
	\end{proof}

	\begin{proof}[Proof of Lemma \ref{l:two filtrations}] 
	For simplicity, we write 
		\begin{equation}\label{eq:h_p}
			h_p^{\pm}(\alpha,\beta):=\dim \gr_{F_{\rr{irr}}}^p\mathrm{H}^{n+m-1}_{\mathrm{dR}}(\bb{G}_m^{n+m-1},\pm f_a)^{(G,\tilde\chi\times \rho)}.
		\end{equation}
	By Lemma~\ref{lem::Newton-weight}, for every $p\in \mathbb{Q}$, we have 
		\begin{equation}\label{eq::inclusion}
			G^p\mathrm{H}^{n+m-1}_{\mathrm{dR}}(\bb{G}_m^{n+m},\pm  f_a)^{(G,\tilde\chi\times \rho)} \subset F^{p}_{\irr}\mathrm{H}^{n+m-1}_{\mathrm{dR}}(\bb{G}_m^{n+m},\pm f_a)^{(G,\tilde\chi\times \rho)},
		\end{equation}
	which implies that $\sum_{q\leq p} \delta^{\pm}_q(\alpha,\beta)\leq \sum_{q\leq p}h_{q}^{\pm}(\alpha,\beta)$.
	
		To prove the reverse inclusion, we consider the duality between the two filtered vector spaces $(\mathrm{H}^{n+m-1}_{\mathrm{dR}}(\bb{G}_m^{n+m-1},\pm f_a)^{(G,\tilde\chi\times \rho)},F^\bullet_{\irr})$ and $(\mathrm{H}^{n+m-1}_{\mathrm{dR}}(\bb{G}_m^{n+m-1},\mp f_a)^{(G,\tilde\chi\inv\times \rho\inv)},F^{\bullet}_{\irr})$, induced by Proposition~\ref{prop::EMHS-properties} and \cite[Thm.\,2.2]{yu2012irregular}. More precisely, we have 
	\begin{equation}\label{eq::duality}
		h^{\pm}_p(\alpha,\beta)=h^{\mp}_{n+m-1-p}(\bar\alpha,\bar \beta).
	\end{equation}
	Combining Lemma~\ref{lem:dual-dimension-G} and the equations \eqref{eq::inclusion} and \eqref{eq::duality}, we see, for any $p\in \mathbb{R}$, that
\begin{equation*}
	\begin{split}
	&\dim G^p \mathrm{H}^{n+m-1}_{\mathrm{dR}}(\bb{G}_m^{n+m},\pm  f_a)^{(G,\tilde\chi\times \rho)} \\
	= & \sum_{q\leq p}\delta^{\pm}_q(\alpha,\beta)
	\leq \sum_{q\leq p}h^{\pm}_q(\alpha,\beta)
	= \sum _{q\geq n+m-1-p}h_{q}^{\mp}(\bar\alpha,\bar \beta)\\
	\leq &\sum _{q\geq n+m-1-p}\delta_q^{\mp}(\bar \alpha,\bar \beta)
	=\sum_{q\leq p}  \delta_q^{\pm}(\alpha,\beta)\\
	=&\dim G^p \mathrm{H}^{n+m-1}_{\mathrm{dR}}(\bb{G}_m^{n+m},\pm  f_a)^{(G,\tilde\chi\times \rho)}.
	\end{split}
\end{equation*}
Hence, both sides in \eqref{eq::inclusion} have the same dimension for every $p$. Then Lemma \ref{l:two filtrations} follows.
\end{proof}

\subsubsection{Proof of Theorem \ref{thm:Hodge-fil}}

\begin{proof} We may assume $\alpha_1=0$ by \eqref{eq:reduce-to-alpha1=0}.
By \cite[Prop.\,3.54]{Sabbah18irregular} and \cite[Prop.\,11.22]{Mochizuki2021}, the irregular Hodge filtration on $\Hyp(\alpha,\beta)$ induces those on fibers $\Hyp(\alpha,\beta)_a$ at closed points of $S$, i.e., $(F^{\bullet}_{\rr{irr}}\Hyp(\alpha,\beta))_a=F^{\bullet}_{\rr{irr}}(\Hyp(\alpha,\beta))_a$. We have shown in Theorem \ref{thm::hodge-number} that the irregular Hodge filtration on the fibers $\Hyp(\alpha,\beta)_a$ are given in terms of the cohomology classes $\omega_{r,s}$ in \eqref{eq:auxiliary-fil}. Hence, we deduce that the irregular Hodge filtration on $\Hyp(\alpha,\beta)$ is the one in assertion (1), and the irregular Hodge numbers are those given in (2).

To prove assertion (3), observe that each element $\omega$ in $F^p_{\rr{irr}}\mathcal{H}^{n+m-1}_{\mathrm{dR}}(U/S,f)^{(G,\tilde\chi\times \rho)}$ is a linear combination of $\omega_{r,s}$ with $n+m-1-w(\omega_{r,s})\geq p$. Since $w(\nabla_{z\partial_z}(\omega_{r,s}))=1+w(\omega_{r,s})$, one has $w(\nabla_{z\partial_z}(\omega))\leq 1+w(\omega)$. So for each $a\in S(\bb{C})$, by the Newton polyhedron filtration \eqref{eq:newton-monomial}, $\nabla_{z\partial_z}(\omega)$ is inside 
	$F^{p-1}_{\rr{irr}}\mathcal{H}^{n+m-1}_{\mathrm{dR}}(U/S,f)^{(G,\tilde\chi\times \rho)}_a.$
Therefore, we conclude that $\nabla(\omega)\in \Omega^1_S\otimes  F^{p-1}_{\rr{irr}}\mathcal{H}^{n+m-1}_{\mathrm{dR}}(U/S,f)^{(G,\tilde\chi\times \rho)}$.
\end{proof}

\section{Frobenius structures on hypergeometric connections and \texorpdfstring{$p$}{p}-adic estimates}
\label{sec:Frob-struc}
	In this section, let $p$ be a prime number and $k=\mathbb{F}_q$ the finite field with $q=p^s$ elements for an integer $s\ge 1$. 
	Let $K$ be a finite extension of $\mathbb{Q}_p$ with residue field $k$ containing an element $\pi$ satisfying $\pi^{p-1}=-p$. 
	We fix such an element $\pi$ and denote the associated additive character by $\psi\colon\mathbb{F}_p\to K^{\times}$ \cite[(1.3)]{Ber84}. 
	The $q$-th power Frobenius on $k$ admits a lift $\sigma=\id$ on $\mathcal{O}_K$. 

	Let $n>m$ be two integers, $\alpha=(\alpha_i=\frac{a_i}{q-1})_{i=1}^n,\beta=(\beta_j=\frac{b_j}{q-1})_{j=1}^m$ be two sequences of non-decreasing rational numbers $\in [0,1)$ with denominator $q-1$. 
	Let $\omega:k^{\times}\to K^{\times}$ be the Teichmüller character and set $\chi_i=\omega^{a_i}$ (resp. $\rho_j=\omega^{b_j}$).
	The hypergeometric sum associated to $\psi, \chi=(\chi_1,\dots,\chi_n), \rho=(\rho_1,\dots,\rho_m)$ is defined, for $a\in k^{\times}$, by 
	\begin{equation}
		\CHyp_{(n,m)}(\chi;\rho)(a)=
		\sum_{\begin{subarray}{c} x_i, y_j\in k^{\times}, \\ x_1\ldots x_n=ay_1\ldots y_{m}\end{subarray}} \psi\biggl( \Tr_{k/\mathbb{F}_p} \biggl(\sum_{i=1}^n x_i -\sum_{j=1}^m y_j\biggr)\biggr) \cdot 
		\prod_{i=1}^n \chi_i(x_i)
		\prod_{j=1}^m \rho_{j}^{-1}(y_j).
	\label{hyp sums}
	\end{equation}
	When $(\chi,\rho)$ is non-resonant, the above sum equals to (up to a sign) the Frobenius trace of the hypergeometric overconvergent $F$-isocrystal $\DHyp(\chi,\rho)$ at $a\in \mathbb{G}_{m}(k)$ \cite{Miy}. 
	Its underlying connection is the hypergeometric connection $\Hyp_{(-1)^{m+np}/\pi^{n-m}}$ \cite[Thm.\,4.1.3]{Miy}. 
	When $(\chi,\rho)$ is resonant, the above sum can be written as a sum of $n$ Weil numbers (see \S~\ref{sss:hypsum-resonant}). 

	We are interested in the $p$-adic valuation of Frobenius eigenvalues (normalized by $\ord_q$) of the above sum (called \textit{Frobenius slopes}), encoded in the Frobenius Newton polygon \cite[\S 2]{Mazur}. 

	Recall that the irregular Hodge numbers of the hypergeometric connection $\Hyp(\alpha;\beta)$ is given by the function $\theta\colon\{1,\dots,n\}\to \mathbb{Q}$ \eqref{eq:theta(k)} is defined as
	\begin{equation} \label{eq:thetaFrob}
			\theta(k)= (n-m)\alpha_k+\#\{i\mid \beta_i< \alpha_k\}+(n-k)-\sum_{i=1}^n \alpha_i+\sum_{j=1}^m\beta_j.
	\end{equation}

\begin{theorem}	\label{t:Frobslope}
	If orders of $\chi_i,\rho_j$ divide $p-1$, then for each $a\in \mathbb{G}_m(k)$, the Frobenius Newton polygon of $\CHyp_{(n,m)}(\chi;\rho)(a)$ coincides with the irregular Hodge polygon defined by \eqref{eq:thetaFrob}. 
\end{theorem}

	A ``Newton above Hodge'' type result for twisted exponential sums was obtained by Adolphson and Sperber \cite{AS93}. In our case, we show that the Hodge polygon in \textit{loc. cit.} for hypergeometric sums coincides with the irregular Hodge polygon of hypergeometric connections. 
	Then, we apply a result of Wan \cite{Wan93} to conclude ``Newton equals to Hodge''. 
	In \cite{XZ22}, the second author and Zhu used a similar argument to study the Newton polygon of Kloosterman sums for classical groups. 
	
\subsection{Frobenius Newton polygon above Hodge polygon}
	In this subsection, we revise Adolphson--Sperber's result on ``Newton above Hodge'' for certain twisted exponential sums \cite{AS93} and study the associated Hodge polygons. 

\begin{secnumber} \label{sss:AS}
Let $N$ be a positive integer, 
        $$\chi=(\chi_1,\dots,\chi_N):(k^{\times})^N\to K^{\times}$$
    a multiplicative character, 
	and $g:\mathbb{G}_{m}^N\to \mathbb{A}^1$ a morphism defined by a Laurent polynomial 
	\[
	g(x_1,\cdots,x_N)=\sum_{j=1}^M a_j x^{u_j}\in k[x_1^{\pm},\cdots,x_N^{\pm}],
	\]
where $\{u_j\}_{j=1}^M$ is a finite subset of $\mathbb{Z}^N$ and $a_j\in k^{\times}$. For $m\in \bb{N}$, we consider the sums
\begin{equation} \label{eq:twistsums}
	S_{m}(\chi,g)=\sum_{x\in \Gm^N(\mathbb{F}_{q^m})} \chi^{(m)}(x)\psi^{(m)}(g(x)),
\end{equation}
	where $\chi^{(m)}=\chi\circ \Nm_{\mathbb{F}_{q^m}/k}$, $\psi^{(m)}=\psi\circ \Tr_{\mathbb{F}_{q^m}/\mathbb{F}_p}$. 
	The associated $L$-function 
	\begin{equation} \label{eq:Lfunction}
	L(\chi,g;T)=\exp\biggl(\sum_{m\ge 1}S_m(\chi,g)\frac{T^m}{m}\biggr)
	\end{equation}
	is a rational function in $T$ by the Grothendieck--Lefschetz trace formula (or Dwork trace formula). 

	Recall that we denote  $\Delta=\Delta(g)$ by the convex closure in $\mathbb{R}^N$ generated by the origin and lattices defined by the exponents $\{u_j\}$ of $g$ in \cref{d:Delta}. 
	Let $C(g)$ be the cone over $\Delta$, i.e., the union of all rays in $\mathbb{R}^N$ emanating from the origin and passing through $\Delta$. 
	Let $M(g)=C(g)\cap \mathbb{Z}^N$.

	Adolphson and Sperber considered a subring $R(g)$ of $k[x_1^{\pm},\dots,x_N^{\pm}]$ defined by
	\[
	R(g)=k[x^{M(g)}].
	\]
	
	We take $d_i\in [0,q-2]$ such that $\chi_i=\omega^{-d_i}$ \footnote{Adolphson--Sperber's convention $\chi_i=\omega^{-d_i}$ is different from our convention in the beginning of \S~\ref{sec:Frob-struc} by a minus sign.}. 
	Let $\overline{d_i}$ be $q-1-d_i$ if $d_i\neq 0$ and $d_i$ if $d_i=0$. 
	We set $\dd=(d_1,\dots,d_N)$, $\overline{\dd}=\{\overline{d_1},\dots,\overline{d_N}\}$, and $N_{\dd}=(q-1)^{-1}\dd+\mathbb{Z}^N$. We define a $R(g)$-module $R_{\dd}(g)$ by
	\[
	R_{\dd}(g)=\bigg{\{} \sum_{\textnormal{finite}} b_u x^u | u\in N_{\dd}\cap C(g),b_u\in k\bigg{\}}. 
	\]

	There exists a (minimal) positive integer $M$ such that for all $u\in \frac{\mathbb{Z}^N}{q-1}\cap C(g)$, the weight function $w(u)$, defined as the minimal positive real number $w$ such that $u\in w\Delta(g)$, is a rational number with denominator dividing $M$. 
	Then $w$ defines an increasing filtration on $R(g)$ by
	\[
	R(g)_{i / M}=\bigg{\{}\sum_{u \in M(g)} b_{u} x^{u}: w(u) \leq \frac{i}{M} \text { for all } u \text { with } b_{u} \neq 0\bigg{\}}.
	\]
	We denote the associated graded module by
	\[
	\overline{R}(g)=\oplus_{i\ge 0} \overline{R}(g)_{i/M},\quad \overline{R}(g)_{i/M}=R(g)_{i/M}/R(g)_{(i-1)/M}.
	\]
	Similarly, we equip $R_{\dd}(g)$ with a filtration compatible with that of $R(g)$, and let $\overline{R}_{\dd}(g)$ be the associated graded $\overline{R}(g)$-module. 
\end{secnumber}
\begin{secnumber} \label{sss:HP}
	In the following, we assume that $g$ is non-degenerate and that $\dim \Delta(g)=N$. 

	For $1\le i\le N$, let $\overline{g}_i$ be the image of $x_i\frac{\partial}{\partial x_i}g$ in $\overline{R}(g)_1$, and set
	\[
	\overline{I}_{\dd}=\overline{g}_1\overline{R}(g)_{\dd}+\dots+ \overline{g}_N \overline{R}(g)_{\dd}
	\]
	a graded submodule of $\overline{R}(g)_{\dd}$. 
	For each $i\ge 0$, we define a finite set $\mathscr{B}_{\dd}^{i/M}(=\mathscr{B}(g)_{\dd}^{i/M})\subset N_{\dd}\cap C(g)$ of exponents as follows. 
	We take a $k$-linearly independent set of monomials $\{x^{\mu}|\mu\in \mathscr{B}_{\dd}^{i/M}\}$ of weight $i/M$ such that the $k$-subspace $\overline{V}_{\dd,i/M}$, which it spans, is a complement to $\overline{R}(g)_{\dd,i/M}\cap \overline{I}_{\dd}$, i.e.,
	\[
	\overline{R}(g)_{\dd,i/M}= \overline{V}_{\dd,i/M}\oplus (\overline{R}(g)_{\dd,i/M}\cap \overline{I}_{\dd,i/M}).
	\]
	We set $\mathscr{B}(g)_{\dd}=\cup_{i\ge 0} \mathscr{B}(g)_{\dd}^{i/M}$ and $V(g)$ the volume of $\Delta(g)$. 
	The quotient $\overline{R}(g)_{\dd}/\overline{I}_{\dd}$ admits a basis of monomials in $S_{\dd}$ and has dimension \cite[Lem.\,2.8]{AS93}
	\[
		\dim \overline{R}(g)_{\dd}/\overline{I}_{\dd} =N! V(g). 
	\]

	In this case, the $L$-function $L(\chi,g;T)^{(-1)^{N-1}}$ \eqref{eq:Lfunction} is a polynomial of degree $N!V(g)$ \cite[Cor.\,2.12]{AS93}. 
	Adolphson and Sperber studied the Frobenius Newton polygon of this $L$-function and compared it with a Hodge polygon defined as below. 

	For an integer $0\le d\le q-2$, let $d'$ be the nonnegative residue of $pd$ modulo $q-1$.
	Recall that $q=p^s$ for an integer $s\ge 1$. 
	For $\dd=(d_1,\dots,d_N)$, we set $\dd'=(d_1',\dots,d_{N}')$ and $\dd^{(i)}$ the $i$-th composition of $(-)'$ on $\dd$ for $i\ge 1$. 
	Note that $\dd^{(s)}=\dd$. 

	We arrange elements of $S_{\dd}=\{ u_{\dd}(1),\dots,u_{\dd}(N!V(g))\}$ by $w(u_{\dd}(1))\le \dots \le w(u_{\dd}(N!V(g)))$. 
	And we repeat this ordering for $S_{\dd'},\dots,S_{\dd^{(s-1)}}$. 
	For an integer $\ell\ge 0$, we set \cite[Thm.\,3.17]{AS93}
	$$
	W(\ell)=\operatorname{card}\bigg{\{} j \ \bigg{|}\  \sum_{i=0}^{s-1} w(u_{\dd^{(i)}}(j))=\frac{\ell}{M}\bigg{\}}.
	$$
	When $\ell> sN M$, we have $W(\ell)=0$.

	The Hodge polygon $\HP(\Delta(g)_{\dd})$ is defined by the convex polygon in $\mathbb{R}^2$ with vertices $(0,0)$ and 
	\[
	\biggl( \sum_{\ell=0}^m W(\ell), \frac{1}{sM}\sum_{\ell=0}^m \ell W(\ell)\biggr),\quad m=0,1,\dots, sNM.
	\]
\end{secnumber}

\begin{thm}[{\cite[Cor.\,3.18]{AS93} }]
	If $g$ is non-degenerate and $\dim(\Delta(g))=N$, the Frobenius Newton polygon of $L(\chi,g;T)^{(-1)^{N-1}}$ lies above the Hodge polygon $\HP(\Delta(g)_{\dd})$, and their endpoints coincide. 
	\label{t:AS}
\end{thm}
\begin{defn} \label{d:ordinary}
We say that $(g,\chi)$ is \textit{ordinary} if these two polygons coincide. 
When the character $\chi$ is trivial, we simply say $g$ is \textit{ordinary}. 
\end{defn}

\begin{secnumber}
	In the following, we apply the above theory to the case of hypergeometric sum at the beginning of \S~\ref{sec:Frob-struc}. 
	We may assume that $\chi_1$ is trivial (i.e. $\alpha_1=0$). 
	Let $a$ be an element of $k^{\times}$. 
	We take $N=n+m-1$, $\dd=(\overline{a_2},\dots,\overline{a_n},b_1,\dots,b_m)$, and $g$ to be the non-degenerate function \eqref{eq::Laurent-polynomial}
	\[
	f_a=a\frac{y_1\dots y_m}{x_2\dots x_n}+ x_2+\dots+x_n -y_1-\dots -y_m. 
	\]
	Then, we recover the hypergeometric sum \eqref{hyp sums} from \eqref{eq:twistsums}. 

\end{secnumber}
\begin{prop} \label{p:TwoHP}
	If $(\chi,\rho)$ is non-resonant and orders of the characters $\chi_i$ and $\rho_j$ divide $p-1$, 
	then the Hodge polygon $\HP(\Delta(f_a)_{\dd})$ coincides with the irregular Hodge polygon associated to $(0,\alpha_2=\frac{a_2}{p-1},\dots,\alpha_n=\frac{a_n}{p-1})$, $(\beta_1=\frac{b_1}{p-1},\dots,\beta_m=\frac{b_m}{p-1})$. 
\end{prop}

\begin{proof}
	Since $\alpha_i,\beta_j$ have denominators dividing $p-1$, the numbers $\dd^{(i)}$ are equal to $\dd$ for every $i\ge 1$. 
	In particular, the multi-set of slopes of $\HP(\Delta(f_a)_{\dd})$ coincides with $w(S_{\dd})=\{\omega(u)|u\in S_{\dd}\}$. 

	The cohomology classes $\omega_{r,\ell}=g_{r,\ell}\cdot \eta$ in Proposition \ref{p:coh basis} form a basis of de Rham cohomology group $\HH^{n+m-1}_{\dR}(U_a, f_a)^{(G,\tilde{\chi}\times \rho)}$. 
	By the calculation of cohomology groups \cite[\S\,3,\,Thm.\,3.14]{AS93}, the functions $\{g_{r,\ell}\}$ 
	also forms a basis of $\overline{V}_{\overline{\dd}}$, with $\overline{\dd}=(a_2,\dots,a_n,\overline{b_1},\dots,\overline{b_m})$. 
	Hence $w(S_{\overline{\dd}})=\{w(g_{r,\ell})|0\le r\le m,1\le \ell \le s_{r+1}-s_r\}$. 
	By \eqref{eq:newton-monomial}, Lemma \ref{lem::Newton-weight} and the duality \eqref{eq::duality}, the set of weights $w(S_{\dd})$ coincides with the set of irregular Hodge numbers \eqref{eq:thetaFrob}. Then the proposition follows. 
\end{proof}

\subsection{Frobenius slopes of hypergeometric sums: proof of \cref{t:Frobslope}} \label{ss:resonantFrob}

\begin{secnumber} \label{sss:hypsum-resonant}
	We prove \cref{t:Frobslope} by induction on $n$. 
	We first show that we can deduce the assertion in the resonant case from the induction hypothesis. 
	Suppose the theorem holds when the rank of the hypergeometric $F$-isocrystal is less than $n$. 

	We slightly modify our convention on $\alpha,\beta$ by replacing those $\alpha_i,\beta_j=0$ by $1$ and then arranging them as $0<\alpha_1\le \alpha_2 \le\dots \le \alpha_n \le 1$ and $0<\beta_1\le \dots \le \beta_m \le 1$. 
	Note that this modification does not change the set $\{\theta(1),\dots,\theta(n)\}$ of irregular Hodge numbers. 
	After twisting a multiplicative character, we may assume that $\chi_n=\rho_m=1$ are the trivial characters (i.e. $\alpha_n=\beta_m=1$).  
	Then we have the following identities:
	\begin{eqnarray} \label{eq:dec-resonant}
		&& \CHyp_{(n,m)}(\chi;\rho)(a) \\ 
		\nonumber
		&=& \sum_{x_i, y_j\in k^{\times}} \psi\biggl(\sum_{i=1}^{n-1} x_i+a\frac{y_1\cdots y_m}{x_1\cdots x_{n-1}}-\sum_{j=1}^m y_j\biggr) \cdot 
		\prod_{i=1}^{n-1} \chi_i(x_i)
		\prod_{j=1}^{m-1} \rho_{j}^{-1}(y_j) \\
		\nonumber
		&=& \sum_{x_i, y_j\in k^{\times},y_m\in k} \psi\biggl( \sum_{i=1}^{n-1} x_i-\sum_{j=1}^{m-1} y_j+y_m(a\frac{y_1\cdots y_{m-1}}{x_1\cdots x_{n-1}}-1)\biggr) \cdot 
		\prod_{i=1}^{n-1} \chi_i(x_i)
		\prod_{j=1}^{m-1} \rho_{j}^{-1}(y_j)  \\
		\nonumber
		&& - \sum_{x_i, y_j\in k^{\times}} \psi\biggl( \sum_{i=1}^{n-1} x_i-\sum_{j=1}^{m-1} y_j\biggr) \cdot 
		\prod_{i=1}^{n-1} \chi_i(x_i)
		\prod_{j=1}^{m-1} \rho_{j}^{-1}(y_j)  \\
		\nonumber
		&=& 
		q\CHyp_{(n-1,m-1)}(\chi';\rho') - (-1)^{m-1} \prod_{i=1}^{n-1} G(\psi,\chi_i)\prod_{j=1}^{m-1} G(\psi,\rho_j^{-1}),
	\end{eqnarray}
	where $\chi'=(\chi_1,\dots,\chi_{n-1})$, $\rho'=(\rho_1,\dots,\rho_{m-1})$, and $G(\psi,\chi_i)=\sum_{x\in k^{\times}}\psi(x)\chi_i(x)$ denotes the Gauss sum. 
In particular, the above sum can be written as a sum of $n$ Weil numbers by induction. 
	
	Let $\theta'$ be the function \eqref{eq:thetaFrob} defined by rational numbers $\alpha_1,\dots,\alpha_{n-1},\beta_1,\dots,\beta_{m-1}$. 
	Then, we have
	\[
	\theta(k)=\theta'(k)+1, ~\forall~ 1\le k\le n-1
	\] 
	and 
	\[\theta(n)=\sum_{i=1}^n (1-\alpha_i) + \sum_{\beta_j< 1} \beta_j = \ord_q \biggl( \prod_{i=1}^{n-1} G(\psi,\chi_i)\prod_{j=1}^{m-1} G(\psi,\rho_j^{-1}) \biggr),
	\]
	where second identity follows from Stickelberger's theorem $\ord_q G(\psi,\omega^{-k})=\frac{k}{p-1}$. 
	Then, the theorem in the resonant case follows from the induction hypothesis and the above decomposition. 
\end{secnumber}

\begin{secnumber}\label{sss:proof-final}
	By the previous argument, we may assume that the assertion for the hypergeometric sum of type $(n,m)$ defined by a resonant pair $(\alpha,\beta)$ is already proved. 
	It suffices to treat the non-resonant case. 
	We may assume $\chi_1=1$ is trivial. 

	We set $\widetilde{f}_a(x_2,\dots,x_n,y_1,\dots,y_m)=f_a(x_2^{p-1},\dots,x_n^{p-1},y_1^{p-1},\dots,y_m^{p-1})$. We first prove the ordinariness of exponential sum associated to $\widetilde{f}_a$ (Definition \ref{d:ordinary}) using a theorem of Wan \cite{Wan93}. 

	Let $\delta_1,\cdots,\delta_{m+n}$ be all the co-dimension $1$ surfaces of $\Delta=\Delta(\widetilde{f}_a)$ which do not contain the origin. Let $\widetilde{f}_a^{\delta_i}$ be the restriction of $\widetilde{f}_a$ to $\delta_i$, which is also non-degenerate. 
	By \cite[Thm.\,3.1]{Wan04}, $\widetilde{f}_a$ is ordinary if and only if each $\widetilde{f}_a^{\delta_i}$ is ordinary. 

	Each Laurent polynomial $\widetilde{f}_a^{\delta_i}$ is diagonal in the sense of \cite[\S\,2]{Wan04}. 
	If $V_1,\cdots,V_{m+n-1}$ denote the vertex of $\delta_i$ written as column vectors, the set $S(\delta_i)$ of solutions of
	\begin{displaymath}
		(V_1,\cdots,V_{m+n-1}) 
		\begin{pmatrix}
			r_1 \\ \vdots \\ r_{m+n-1}
		\end{pmatrix}
		\equiv 0 ~(\textnormal{mod } 1),\quad r_i \textnormal{ rational, } 0\le r_i<1,
	\end{displaymath}
	forms an abelian group, which is isomorphic to $(\mathbb{Z}/ (p-1) \mathbb{Z})^{n+m-1}$. 
	We deduce that for each $\delta_i$, $\widetilde{f}_a^{\delta_i}$ is ordinary by \cite[Cor.\,2.6]{Wan04}. 

	We have a decomposition of exponential sums as follows:
	\begin{equation} \label{eq:dec-sum}
	\sum_{x_i,y_j\in k^{\times}}\psi(\widetilde{f}_a(x_i,y_j)) = 
	\sum_{\chi_i,\rho_j} \CHyp_{(n,m)}(\chi,\rho)(a),
	\end{equation}
	where the sum is taken over all multiplicative characters $\chi_i,\rho_j$ with $2\le i\le n, 1\le j\le m$ of orders dividing $p-1$. 
	We have a similar decomposition for $\mathscr{B}_{\dd}$ (\S~\ref{sss:HP}) given by
	\[
	\mathscr{B}_{1}(\widetilde{f}_a) =\bigsqcup_{\dd} \mathscr{B}_{\dd}(f_a),
	\]
	where $1=(0,0,\dots,0)$ and $\dd$ is taken over all $(n+m-1)$-tuple of rational numbers with denominators $p-1$ in $[0,1)$.  

	On the left-hand side of \eqref{eq:dec-sum}, we have shown ``Newton equals to Hodge'' (i.e. the ordinariness of $\widetilde{f}_a$). 
	Together with the ``Newton above Hodge'' for each hypergeometric sum (\cref{t:AS}), we deduce that ``Newton equals to Hodge'' for each component of the right-hand side. 
	Then the assertion in the non-resonant case follows from Proposition \ref{p:TwoHP}.
	\hfill \qed
\end{secnumber}

In particular, our proof shows Proposition \ref{p:TwoHP} in the resonant case. 

\begin{cor}
	Proposition \ref{p:TwoHP} holds without the non-resonant assumption.
\end{cor}

\begin{proof}
	In the resonant case, the Frobenius Newton polygon equals the irregular Hodge polygon by \S \ref{sss:hypsum-resonant}. 
	By the proof in \S \ref{sss:proof-final}, the Frobenius Newton polygon equals the Hodge polygon defined by Adolphson--Sperber. Then the assertion follows. 
\end{proof}

\subsection*{Acknowledgement}

The authors thank Javier Fresán, Claude Sabbah, and Christian Sevenheck for their valuable discussions.
YQ acknowledges the financial support from the European Research Council (ERC) under the European Union's Horizon 2020 research and innovation program (Grant agreement n° 101020009) for part of this work. DX acknowledges the financial support from the National Natural Science Foundation of China Grants (Nos. 12222118 and 12288201) and the CAS Project for Young Scientists in Basic Research, Grant No. YSBR-033.
Part of the work was done when YQ was staying at Morningside Center of Mathematics, and he would like to thank Morningside Center of Mathematics for its hospitality. 

\bibliography{hyp.bib}
\bibliographystyle{abbrv}

\end{document}